\documentclass[reqno]{amsart}

\newcommand{\dd}{\mathrm{d}}

\usepackage{array,multirow}
\usepackage{graphicx} 
\allowdisplaybreaks

\usepackage[ruled,linesnumbered]{algorithm2e}[section]
\SetKwInOut{Constants}{Constants}

\newtheorem{lemma}{Lemma}[section]
\newtheorem{theorem}{Theorem}[section]
\newtheorem{claim}{Claim}[section]

\newtheorem*{theorem*}{Theorem}

\theoremstyle{definition}

\theoremstyle{remark}
\newtheorem{remark}{Remark}[section]

\usepackage{hyperref}
\usepackage{cleveref}
\usepackage{xcolor}

\title{Fast expansion into harmonics on the ball}

\author[J. Kileel]{Joe Kileel}
\address{Department of Mathematics, University of Texas at Austin, Austin, TX 78712 USA}
\email{jkileel@math.utexas.edu}

\author[N.F. Marshall]{Nicholas F. Marshall}
\address{Department of Mathematics, Oregon State University, Corvallis, OR 97330 USA}
\email{marsnich@oregonstate.edu}

\author[O. Mickelin]{Oscar Mickelin}
\address{Program in Applied and Computational Mathematics, Princeton University, Princeton, NJ 08540 USA}
\email{hm6655@princeton.edu}

\author[A. Singer]{Amit Singer}
\address{Department of Mathematics and Program in Applied and Computational Mathematics, Princeton University, Princeton, NJ 08540 USA}
\email{amits@math.princeton.edu}

\thanks{Code accompanying this paper is available at \url{https://github.com/oscarmickelin/fle_3d}}

\usepackage[inner=1.6in,outer=1.6in,top=1.3in,bottom=1.3in,
marginparwidth=.7in,marginparsep=.1in]{geometry}

\begin{document}

\begin{abstract}
We devise fast and provably accurate algorithms to transform between an
$N\times N \times N$ Cartesian
voxel representation of a three-dimensional function
and its expansion into the {ball harmonics}, that is, the eigenbasis of the Dirichlet Laplacian on the unit ball in $\mathbb{R}^3$. Given $\varepsilon > 0$, our algorithms achieve relative $\ell^1$ - $\ell^\infty$ accuracy $\varepsilon$ in time $\mathcal{O}(N^3 (\log N)^2 + N^3 |\log \varepsilon|^2)$, while the na\"{i}ve direct application of the expansion operators has time complexity $\mathcal{O}(N^6)$. We illustrate our methods on numerical examples.
\end{abstract}



\maketitle

\section{Introduction}
We consider the problem of expanding discretized functions $f:[-1,1]^3 \rightarrow \mathbb{C}$ supported on the unit ball $\mathbb{B} = \{ x \in \mathbb{R}^3 : \|x\|_{ \ell^2} \leq 1\}$
into the Dirichlet Laplacian eigenfunctions on the unit ball. 
Given samples of the function $f$ on an $N \times N \times N$ Cartesian grid (we refer to such a discretization as a \emph{volume}), we seek a fast method to obtain the corresponding eigenbasis expansion coefficients. Conversely, given its expansion coefficients we also want a fast method to evaluate the corresponding function $f$ on the grid.
The Dirichlet Laplacian eigenfunctions can be written in spherical coordinates $(r,\theta,\phi)$ as
\begin{equation}\label{eq:sph_bessel1}
\psi_{k, \ell, m}(r, \theta , \phi) = c_{\ell k}j_\ell(\lambda_{\ell k} r)  Y_{\ell}^m(\theta,\phi)
\chi_{[0,1)}(r),
\end{equation}
for $ m \in \{-\ell,\ldots,\ell\}$, $ \ell \in \mathbb{Z}_{\ge 0}$, and $ k \in \mathbb{Z}_{>0}$, where the $c_{\ell k}$ are positive normalization constants,
$j_\ell$ are spherical Bessel functions of the first kind,
 $\lambda_{\ell k}$ is the $k$-th positive root of $j_\ell$,
 $Y_{\ell}^m(\theta,\phi)$ are spherical harmonics and $\chi_{[0,1)}$ is an indicator function for $[0,1)$, which extends the functions to $\mathbb{R}^3$; for precise definitions, 
see Appendix~\ref{notation}. 
These eigenfunctions are solutions to the eigenvalue problem
\begin{align}
-\Delta \psi &= \lambda^2 \psi \, \quad 
  \text{in } \mathbb{B}, \\
\psi &= 0 \qquad \,\, \text{on } \partial \mathbb{B},
\end{align}
where $\Delta$ is the Laplacian, see \S \ref{sec:lapeigedirdef}. These functions are the \emph{ball harmonics} in the sense that they are the standing waves associated with the resonant frequencies of the ball with fixed Dirichlet boundary conditions.

The ball harmonics
have a number of beneficial numerical properties, namely they are \emph{orthonormal}, \emph{frequency-ordered}, and \emph{steerable}.
Orthonormality stems from the eigenfunction formulation and the fact that the Laplacian with Dirichlet boundary conditions is a self-adjoint operator.  
The frequency-ordering property arises by ordering the eigenfunctions according to their real eigenvalue.
It allows one to perform low-pass filtering on volumes expanded into the basis, by retaining the expansion coefficients corresponding to the lowest frequencies and setting the remaining coefficients to zero. Thirdly, the steerability property means that rotations of the basis functions can be efficiently obtained as closed-form linear combinations of the basis functions via Wigner D-matrices \cite[Eq.~(9.49)]{chirikjian2016harmonic}. This allows for lossless rotation of volumes expanded into the basis. We illustrate the low-passing property in Figure~\ref{fig:downsampling}.

\begin{figure}
    \centering
    \includegraphics[width=0.95\textwidth]{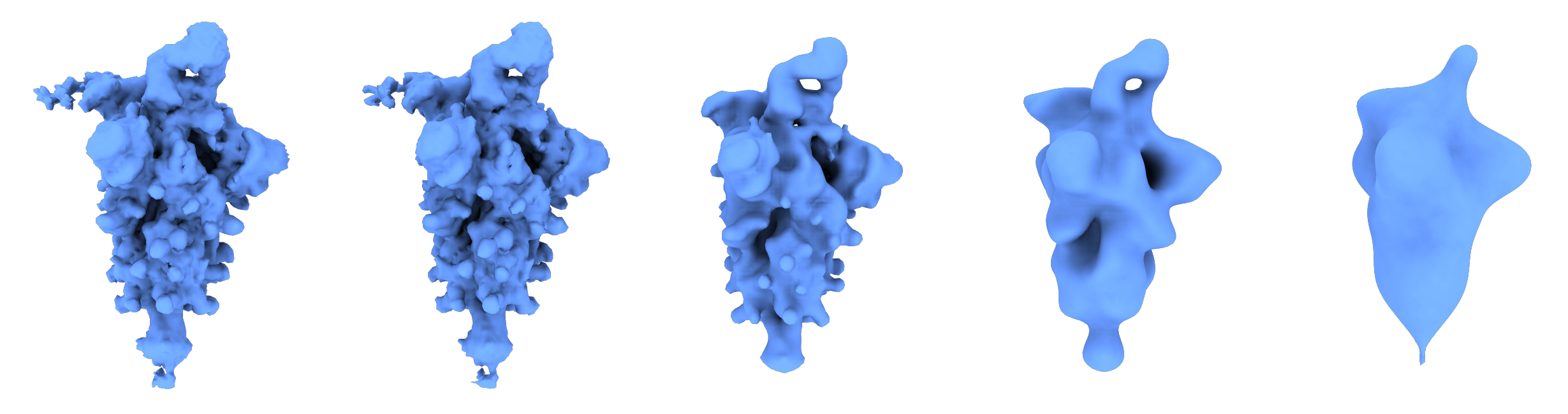}
    \caption{Illustration of low-pass filtering in the ball harmonics basis $\psi_{k,\ell,m}$ for volumes of size $128\times 128 \times 128$. The ground truth volume (leftmost panel) is a 3D density map of the SARS-CoV-2 Omicron spike glycoprotein complex  \cite{guo2022structures} downloaded from the online electron microscopy data bank~\cite{lawson2016emdatabank}. Subsequent panels expand the ground truth volume in ball harmonics and decrease the number of basis functions by dividing the bandlimit by successive factors of $2$, so by retaining the basis functions with corresponding $\lambda_{\ell k}$ at most 201.06, 100.53, 50.21, 25.10, respectively, and use $564641$, $69545$, $8253$ and $1007$ basis functions, respectively. The volumes are rendered using UCSF ChimeraX \cite{goddard2018ucsf}.
    }
    \label{fig:downsampling}
\end{figure}

\subsection{Directly related works}
Separable basis functions of the form in \eqref{eq:sph_bessel1} with spherical harmonics being the angular components play a crucial role in several imaging techniques, such as cryo-electron microscopy and fluctuation X-ray scattering for 3D reconstruction of molecules \cite{kam1980reconstruction,donatelli2015iterative,bendory2023autocorrelation}. In these inverse problems, experimentally computable quantities provide partial information about the expansion of the solution into such a basis. These hinge on using a basis that is closed under rotations, which by the Peter-Weyl theorem requires angular components represented by spherical harmonics \cite[Thm.~8.13]{chirikjian2016harmonic}. These basis functions are also a natural choice in other three-dimensional data processing applications, such as geophysics \cite{lowrie2020fundamentals}, quantum mechanics \cite{banerjee2015spectral}, graph neural networks \cite{gasteiger2020directional}, neuroimaging \cite{galinsky2014automated}, the solution of partial differential equations \cite{gumerov2009broadband}, cosmology \cite{abramo2010cmb,lanusse2012spherical,leistedt20123dex,rassat20123d}, certain types of classification problems \cite{wang2009rotational}, acoustical applications generating personalized spatial sound \cite{politis2016applications}, and rotational alignment of three-dimensional objects \cite{kovacs2002fast}. In such applications, however, there is freedom in choosing the radial basis functions, and several possible candidates exist. They include Laguerre basis functions \cite{leistedt2012exact,mcewen2013fourier,leistedt2013flaglets,wulker2020fast}, Chebyshev polynomials \cite{hollerbach2000spectral}, and prolate spheroidal wave functions (also known as Slepian functions) \cite{sharon2020method}.

In many prior works and  applications, it is assumed that samples of the 3D function $f$ are available on a spherical grid. Since the basis functions in \eqref{eq:sph_bessel1} are separable on a spherical grid, product quadrature can then be used to efficiently compute basis coefficients. Importantly, the setting of the present paper is different, in that we assume access to samples of $f$ only on a Cartesian grid. The separability of the basis functions can therefore not be exploited a priori. That said, the case of radial Laguerre functions (instead of spherical Bessel functions as in this paper) on an unstructured (and not necessarily spherical) grid was previously studied  \cite{wulker2020fast}. In our setting of an $N\times N \times N$ Cartesian grid, the computational complexity of the approach of \cite{wulker2020fast} would be $\mathcal{O}( N^{(6-3/e)}) \approx \mathcal{O}( N^{4.9})$ (see \cite[Thm.~7.1 and the last row of Table 1]{wulker2020fast}), which is significantly higher than that of this article. Overall we are not aware of  any provably accurate expansion algorithms on Cartesian grids with time complexity comparable to the algorithms of this paper.

To our best knowledge, the  approach taken in this paper is most closely related to the techniques of \cite{salehin2010photoacoustic,akramus2012frequency}. However, those papers differ from ours in critical aspects.  Specifically they: i) deal with a different problem by solving inverse problems arising from the inhomogeneous Helmholtz equation; ii) assume access to functions on a spherical grid; and iii) are heuristic in nature and do not prove accuracy guarantees of their algorithms. Indeed, besides operating on a Cartesian grid, a major contribution of this article is in proving accuracy guarantees of our basis expansion method. 
This is done by showing that certain analytical identities for continuous functions $f$ have discrete counterparts when using only samples of $f$. We show that these discrete counterparts can be interpreted as computing inner products of the samples with all basis functions and, moreover, that they can be evaluated quickly and accurately. The approach in this paper may also have extensions to establishing rigorous fast methods for  \cite{leistedt2012exact,mcewen2013fourier,leistedt2013flaglets}, which use some similar analytical formulas.

All together the present article thus provides a rigorous argument for choosing the radial basis functions in \eqref{eq:sph_bessel1}, because we contribute fast and accurate basis expansion algorithms for the choice of ball harmonics. As mentioned, the methods have precise accuracy guarantees, and they are linear in the input size (up to logarithmic factors). 

In passing, we mention that there are other approaches to function representations on the ball not based on spherical harmonics, which do not enjoy the closure under arbitrary rotations property due to the Peter-Weyl theorem.
These include products of Chebyshev polynomials in the radial direction and Fourier modes in the two spherical angular variables in the so-called Double Fourier Sphere method, see \cite{boulle2020computing} for more.

\subsection{Connection to 2D work} The three-dimensional functions in \eqref{eq:sph_bessel1} are closely related to a certain set of basis functions in two variables: the eigenfunctions of the Dirichlet Laplacian on the unit disk in $\mathbb{R}^2$. A subset of the authors previously devised fast algorithms for basis expansion into this two-dimensional basis \cite{marshall2023fast}. The techniques of the present article extend those in \cite{marshall2023fast} to 3D.
In doing so, this paper overcomes several significant technical challenges new to the 3D case:
\begin{itemize}
\item Firstly, in 2D, the accuracy analysis uses well-known results about the uniform quadrature rule on the circle. In contrast, in 3D, the accuracy analysis involves careful estimates for quadrature grids on the sphere, and a number of special function identities.
The error estimates are critical for determining a sufficient number of discretization points to guarantee the 3D algorithms are both fast and accurate.
\item Secondly, in 2D, the non-radial computation is handled by the 
Fast Fourier Transform, which is algebraic, exact and fast in practice.  This simplifies both the error analysis and practical implementation. 
In contrast, in 3D, the algorithms herein use fast approximate spherical harmonics transforms, which introduce additional sources of error as well as implementation tradeoffs between asymptotically fast versus practically fast implementations.
\item Thirdly, the increased dimensionality of the problem in 3D compared to 2D poses implementation and computational complexity challenges.
\end{itemize}

\subsection{Conventions}
Our work includes expository supplementary material, which precisely describes the notation used in the main paper. 
While we try to use standard notation throughout, conventions for special functions and operators vary across mathematical communities, which can lead to ambiguity.
In particular, Appendix~\ref{notation} sets conventions for spherical coordinates in \S \ref{sec:sphericalcoor}, 
spherical Bessel functions in \S \ref{sec:spherebessel},
spherical harmonics in \S \ref{sec:sphereharmonic}, 
ball harmonics in \S \ref{sec:lapeigedirdef},
quadrature rules on $\mathbb{S}^2$ in \S \ref{Clenshaw-Curtis-sec},
spherical harmonic transforms in \S \ref{sec:spherehartransform},
and
fast spherical harmonic transforms in \S 
\ref{sec:fastsphehartransform}.

\section{Main result} \label{sec:main-informal}
{
\subsection{Setup}\label{sec:setup}
Let $f= (f_{j_1,j_2,j_3})$ be an $N \times N \times N$ volume whose entries represent function values at the points
$$
x_{j_1 j_2 j_3} = (hj_1 -1, hj_2 - 1, hj_3 -1), \quad \text{where} \quad h = 1/\lfloor (N+1)/2 \rfloor,
$$
for $j_1,j_2,j_3 \in \{0,\ldots,N-1\}$. More precisely, we assume there is an underlying function (denoted by the same letter, by abuse of notation) ${f} : [-1,1]^3 \rightarrow \mathbb{C}$, which is supported on the unit ball $\mathbb{B}$, such that $f_{j_1 j_2 j_3} = {f}(x_{j_1 j_2 j_3})$.  Let 
\begin{equation} \label{eq:f-list}
x_1, \ldots , x_V \quad \text{ and } \quad f_1, \ldots , f_V
\end{equation}
be an enumeration of $(x_{j_1 j_2 j_3})$ and $(f_{j_1 j_2 j_3})$, respectively, where $V = N^3$ is the number of voxels in the volume $f$. For a given bandlimit $\lambda >0$, let 
\begin{equation}\label{eq:psi-list}
\psi_1,\ldots,\psi_n
\end{equation}
be an enumeration of the functions $\{ \psi_{k,\ell,m} : \lambda_{k,\ell,m} \le \lambda\}$, where the ball harmonics $\psi_{k,\ell,m}$ and constants $\lambda_{k,\ell,m}$ introduced in \eqref{eq:sph_bessel1} are precisely detailed in \S \ref{notation}, and the choice of the bandlimit $\lambda > 0$ is discussed in \S \ref{sec:bandlimit}.

\subsection{Ball harmonics transforms}
We are interested in devising fast and accurate methods for applying the linear operator $B: \mathbb{C}^n \rightarrow \mathbb{C}^V$ defined by
\begin{align}
    (B \alpha)_j &= \sum_{i=1}^n \alpha_i \psi_i(x_j)h^{3/2}, \quad \text{ for } j=1, \ldots ,V, \label{eq:defB}
\end{align}
which maps expansion coefficients to voxelated volumes, and its adjoint $B^*: \mathbb{C}^V \rightarrow \mathbb{C}^n$ defined by
\begin{align}
    (B^* f)_i &= \sum_{j=1}^V f_j \overline{\psi_i(x_j)}h^{3/2}, \quad \text{ for } i=1, \ldots ,n,\label{eq:defB*}
\end{align}
 which maps voxelated volumes to expansion coefficients. 

 \begin{remark}
In \eqref{eq:defB} and \eqref{eq:defB*}, the normalization constant $h^{3/2}$ is included to make the operator $B^*B$ close to the identity. Observe that
$$
(B^*B)_{i,j} = \sum_{k=1}^{V} \overline{\psi_i(x_k)} \psi_j(x_k) h^3.
$$
Since $h^3$ is the volume of a voxel, the right-hand side is a Riemann sum for the inner product between $\psi_i$ and $\psi_j$. Therefore, since the functions $\psi_i$ are orthonormal, it approximately equals $1$ when $i=j$ and $0$ otherwise; more precisely, if the number of basis functions $n$ is fixed, then $B^* B \to I$ as $V \rightarrow \infty$. See Remark~\ref{rmkbandlimith} and Figure~\ref{fig:condb} for further discussion.
\end{remark}

\subsection{Main analytic result}\label{subsec:main-inf}

Our main analytic result describes the computational complexity of new fast algorithms 
(see Algorithms~\ref{algoBH} and \ref{algoB} below) for applying the operators $B$ and $B^*$. The main result is informally stated as follows.

\begin{theorem*}[Informal Statement] 
Assume $n = \mathcal{O}(V)$.
Using the algorithms described in this paper, we can apply the operators $B$ and $B^*$ defined in \eqref{eq:defB} and \eqref{eq:defB*} with relative error $\varepsilon$ 
in  $\mathcal{O}(V(\log V)^2 + V |\log \varepsilon |^2) $ operations.
\end{theorem*}

See Theorem \ref{thm:main} for a precise statement of this result. The fast algorithms for applying $B^*$ and $B$ are detailed in   Algorithms~\ref{algoBH} and \ref{algoB}, respectively.

\section{Algorithms}\label{sec:algorithms}
This section presents the paper's fast algorithms for applying approximations to the operators $B^*$ and $B$ in Algorithms~\ref{algoBH}--\ref{algoB}, respectively. 
We start by giving a precise statement of our main analytic result, which provides error guarantees and the computational complexity for the algorithms.

\begin{theorem}\label{thm:main}
Assume $\lambda \leq 6^{1/3}\pi^{2/3} \lfloor (V^{1/3}+1)/2 \rfloor$ and $\varepsilon >0$ is a user-specified  parameter satisfying $ |\log_2 \varepsilon| \leq 5.3V^{1/3}$. Then Algorithms~\ref{algoBH}--\ref{algoB} take $\mathcal{O}(V(\log V)^2 + V |\log \varepsilon |^2) $ operations and their outputs  $\widetilde{B}^* f$ and 
$\widetilde{B} \alpha $ satisfy the accuracy bounds 
\begin{equation}\label{eq:main-bound}
    \| \widetilde{B}^*f - B^*f\|_{\ell^\infty} \leq \varepsilon \|f\|_{\ell^1} \quad \text{ and } \quad  \,   \| \widetilde{B}\alpha - B\alpha\|_{\ell^\infty} \leq \varepsilon \|\alpha\|_{\ell^1} ,
\end{equation}
for all inputs $f$ and $\alpha$. 
\end{theorem}
The proof of Theorem~\ref{thm:main} is in Appendix~\ref{appendix:proofs}. 
Several immediate remarks are helpful.

\begin{remark}[Other grids] \label{rem:grid}
We remark that the proof of \Cref{thm:main} carries through for points $x_j$ and samples $f_j$ on arbitrary grids, provided $\|x_j\|_{\ell^\infty} \leq 1$. Algorithmically, this is achieved by using a type 3 
(nonuniform to nonuniform) NUFFT transform (see \cite{barnett2019parallel} and \cite[Chapter~7.3]{plonka2018numerical}) in Step 1 and 3 of Algorithms~\ref{algoBH} and \ref{algoB}, respectively.
\end{remark}

\begin{remark}[Choice of norms]
Bounds of the type in \eqref{eq:main-bound} using $\ell^1$ - $\ell^\infty$ bounds are common in the analysis of algorithms that use the non-uniform fast Fourier transform (NUFFT), see for example
\cite[Eq.~(9)]{barnett2021aliasing}.
Since the algorithms of this paper use NUFFT, we obtain the  $\ell^1$ - $\ell^\infty$ form of \eqref{eq:main-bound}. However, we note that Theorem~\ref{thm:main} also ensures accuracy with $\ell^2$ - $\ell^2$ bounds with the same asymptotic computational complexity. Since $$\frac{\| \widetilde{B}^*f - B^*f\|_{\ell^2}}{\|f\|_{\ell^2}} \leq \sqrt{nV}\frac{\| \widetilde{B}^*f - B^*f\|_{\ell^\infty}}{\|f\|_{\ell^1}},$$ running  Algorithms~\ref{algoBH}--\ref{algoB} with accuracy parameter $\varepsilon/\sqrt{nV}$ ensures accuracy $\| \widetilde{B}^*f - B^*f\|_{\ell^2} \leq \varepsilon \|f\|_{\ell^2}$ in time $\mathcal{O}(V(\log V)^2 + V |\log \varepsilon/\sqrt{nV} |^2) = \mathcal{O}(V(\log V)^2 + V |\log \varepsilon |^2) $, using the fact that $n = \mathcal{O}(V)$ with the choice of bandlimit in Theorem~\ref{thm:main} (see \S\ref{sec:bandlimit}).
We emphasize that while this argument that converts $\ell^1-\ell^\infty$ bounds into $\ell^2 - \ell^2$ bounds works in theory, in practice when using finite precision arithmetic, it may lead to overly optimistic estimates when the accuracy parameter $\varepsilon$ is close to machine precision. Indeed, in this case 
$\varepsilon / \sqrt{n V}$ may be below machine precision, making the conversion fail in computer arithmetic with insufficiently high precision arithmetic. In \S \ref{sec:accuracy}, we report both $\ell^1-\ell^\infty$ and $\ell^2-\ell^2$ error estimates.

\end{remark}

\begin{remark}[Dense matrix operators]
When the number of basis functions $n$ is on the order of the number of voxels $V$,
a direct application of $B$ or $B^*$ requires forming a matrix with order $V^2$ elements and then matrix-vector products involve order $V^2$ operations. For example, if a $100 \times 100 \times 100$ volume is given, then $V=10^6$. Storing a dense matrix with $V^2 = 10^{12}$ entries using single precision floating point numbers would require $4$ terabytes of storage.
If this matrix cannot fit in memory, it either needs to be constructed online or moved in and out of memory, creating additional implementation challenges. 
\end{remark}

\begin{remark}[No smoothness assumptions]\label{rem:no_smoothness}
     Theorem \ref{thm:main} does not make any smoothness assumptions on the input function $f$. Theoretically speaking, to ensure that the basis coefficients are finite in the limit as $V\rightarrow \infty$, the underlying function should be at least square-integrable on the unit ball, which practically speaking does not restrict the class of volumes that can be considered.
     
     The fact that the algorithm works for noisy inputs is essential for applications such as low-pass filtering or 
    cryo-electron tomography, where measurement data is heavily corrupted by noise, so we cannot even assume the continuity of the underlying function $f$.  Even in the absence of noise, data is not necessarily smooth, due e.g. to sharp boundaries, corners, or cusps. 
    We note however that if the input $f_{j_1,j_2,j_3}$ were to consist of samples from a sufficiently smooth function, then a fast algorithm could be created by interpolating the volume data $f_{j_1,j_2,j_3}$ from the given Cartesian grid onto a product grid of the form $\{r_1, \ldots , r_R\}\times \{\theta_1 , \ldots , \theta_T\} \times \{\phi_1 , \ldots , \phi_P\}$ in the radial and spherical variables. Because of the separable structure of the basis functions in \eqref{eq:sph_bessel1}, a product-structured quadrature rule could then be applied to compute the basis coefficients with reduced complexity. Further the product quadrature could be computed using fast spherical harmonic transforms in the spherical variables. Assuming $R,T,P$ are $\mathcal{O}(N)$, then the resulting computational complexity would be the same as our method, that is $\mathcal{O}(N^3 (\log N)^2)$. 
    That said, the behavior of this algorithm, especially for noisy volumes, would be unpredictable. 
 The final accuracy of the basis coefficients would be related to the accuracy of the interpolation step and, therefore, to the smoothness of the underlying function $f$ that provides the discretizations $f_{j_1,j_2,j_3}$
 and noise.  
 Hence this straightforward algorithm is of limited utility to various applications of interest. 
 In contrast, the methods described in this paper are simply fast ways to apply the matrix operators $B$ and $B^*$ defined in  \eqref{eq:defB} and \eqref{eq:defB*}, respectively, and their $\ell^1$-$\ell^\infty$ accuracy guarantees hold for all inputs.
\end{remark}

\begin{remark}(Real-valued basis transform)
Statements analogous to Theorem~\ref{thm:main} hold for the real-valued basis functions of \eqref{eq:def_real_psiklm}. Our code provides the option of choosing between the real- and complex-valued bases, as well as converting computed basis coefficients between real and complex form via the transformation in \eqref{eq:def_real_psiklm}.
\end{remark}

\subsection{Key analytic identities}\label{sec:analytic_identities}
The key idea behind the algorithms is the analytic identities described in this section.
With the notation from \eqref{eq:def_spherical_gamma}, we write the basis functions $\psi_{k,\ell , m}$ in \eqref{eigenfundef} as 
\begin{equation}\label{eq:def_eigenf}
\psi_{k, \ell, m}(x) = c_{\ell k} j_\ell(\lambda_{\ell k} r_x) Y_{\ell}^m(\gamma_x) \chi_{[0,1)}(r_x).
\end{equation}
We will likewise write the frequency variable $\omega$ for the Fourier-transform of $\psi_{k, \ell, m}$ in spherical coordinates $\omega = r_\omega \gamma_\omega$, where $r_\omega \geq 0$ and $\gamma_\omega \in \mathbb{S}^2$, see
\S \ref{sec:sphericalcoor}.

We are studying expansions into the ball harmonics, i.e., into the basis defined by \eqref{eq:def_eigenf}, for $ m \in \{-\ell,\ldots,\ell\}$, $ \ell \in \mathbb{Z}_{\ge 0}$, and $ k \in \mathbb{Z}_{>0}$. 
Our fast algorithms build on the following identity
\begin{equation}\label{eq:plane_wave_expansion}
e^{\imath\omega \cdot x} = 4\pi  \sum_{\ell'=0}^\infty \sum_{m'=-\ell'}^{\ell'} \imath^{\ell'} j_{\ell'}(r_x r_\omega) Y_{\ell'}^{m'}(\gamma_x)\overline{Y_{\ell'}^{m'}(\gamma_\omega)} ,
\end{equation}
which is known as Lord Rayleigh's plane-wave expansion \cite[\S 4.7.4]{chirikjian2016harmonic}. The following result is a straightforward consequence of the plane-wave expansion. 
\begin{lemma}
Writing $\omega = r_\omega \gamma_\omega$ as in \eqref{eq:def_spherical_gamma}, the Fourier transform of $\psi$ satisfies
\begin{equation}
\widehat{\psi}_{ k, \ell, m}(\omega) =\left( 4\pi (-\imath)^\ell \int_0^1 c_{\ell k} j_{\ell}(\lambda_{\ell k}r_x)   j_\ell(r_x r_\omega)  r_x^2\dd r_x \right) Y_\ell^m(\gamma_{\omega}).\label{eq:fourier_psi}
\end{equation}
\end{lemma}

\begin{proof}
Using \eqref{eq:plane_wave_expansion} in the definition of the Fourier transform results in
\begin{align*}
 &\widehat{\psi}_{ k, \ell, m}(\omega) = \int_{\mathbb{R}^3} \!\! \psi_{ k, \ell, m}(x) e^{-\imath\omega \cdot x}\dd x \\
 &= 4\pi \int_{\mathbb{R}^3} \!\! \psi_{ k, \ell, m}(x)  \left( \sum_{\ell'=0}^\infty \sum_{m'=-\ell'}^{\ell'} (-\imath)^{\ell'} j_{\ell'}(r_x r_\omega) \overline{Y_{\ell'}^{m'}(\gamma_x)}Y_{\ell'}^{m'}(\gamma_{\omega}) \right)\dd x \\
 &= 4\pi  \! \int_{\mathbb{R}^3} \!\!\! c_{\ell k} j_{\ell}(\lambda_{\ell k}r_x) Y_{\ell}^{m}(\gamma_x) \chi_{[0,1)}(r_x) \\
 &\qquad \qquad \qquad \times \sum_{\ell'=0}^\infty \sum_{m'=-\ell'}^{\ell'}  \!\!\!\! (-\imath)^{\ell'} j_{\ell'}(r_x r_\omega) \overline{Y_{\ell'}^{m'}(\gamma_x)}Y_{\ell'}^{m'}(\gamma_{\omega})\dd x \\
 &= \left( 4\pi (-\imath)^\ell \int_0^1 c_{\ell k} j_{\ell}(\lambda_{\ell k}r_x)   j_\ell(r_x r_\omega)  r_x^2\dd r_x \right) Y_\ell^m(\gamma_{\omega}),
 \end{align*}
where the last step used the orthogonality of the spherical harmonics from \eqref{eq:orthogonality_sph_harm}. We note that interchanging the integral over the sphere with the summations in the calculations is justified since $j_\ell(z)$ decays super exponentially when $z$ is fixed and $\ell \rightarrow \infty$ (see 
\eqref{jnJN} and \cite[Eq. (10.19.1)]{dlmf}), so Fubini's Theorem applies.
\end{proof}
The lemma is used to prove the next theorem, which is the basis of our fast algorithms.  
\begin{theorem}\label{thm:main_identity}
Let $f = \sum_{( k,\ell,m) \in \mathcal{I}} \alpha_{k, \ell, m} \psi_{k, \ell, m}$ for a finite index set $$\mathcal{I}  \subseteq \left\{ (k, \ell, m) : m \in \{-\ell,\ldots,\ell\}, \ell \in \mathbb{Z}_{\ge 0},  k \in \mathbb{Z}_{>0}\right\},$$ and define
\begin{equation}\label{eq:analytic_beta_equation}
    \beta_{\ell, m}(\rho) = \frac{\imath^\ell}{4\pi} \int_{\mathbb{S}^2} \widehat{f}(\rho \gamma_\omega) \overline{Y_\ell^m(\gamma_{\omega})} \dd\sigma(\gamma_\omega).
\end{equation}
It then holds that
\begin{equation}\label{eq:final_identity}
    \alpha_{k, \ell, m} = c_{\ell k}\beta_{ \ell, m }(\lambda_{\ell k}), \quad \text{for all} \quad (k,\ell,m) \in \mathcal{I}. 
\end{equation}
\end{theorem}

\begin{proof}
Using \eqref{eq:fourier_psi}, we obtain
\begin{equation}
\begin{split}
        &\beta_{\ell, m}(\rho) = \frac{\imath^\ell}{4\pi}\int_{\mathbb{S}^2} \sum_{(k',\ell',m') \in \mathcal{I}} \alpha_{k',\ell',m'}\widehat{\psi}_{k',\ell', m'}(\rho \gamma_\omega) \overline{Y_\ell^m(\gamma_{\omega})} \dd \sigma(\gamma_\omega) \\
    &=  \sum_{(k',\ell',m') \in \mathcal{I}} \alpha_{k',\ell',m'} \delta_{\ell \ell'} \delta_{m m'} \int_0^1 c_{\ell k'} j_{\ell}(\lambda_{\ell k'}r_x) j_\ell(r_x \rho)  r_x^2 \dd r_x .
    \end{split}
\end{equation}
By orthogonality of the spherical Bessel functions, it follows that
\begin{equation}
\begin{split}
\beta_{\ell, m}(\lambda_{\ell k}) &= \sum_{(k',\ell',m') \in \mathcal{I}} \alpha_{k',\ell',m'} \delta_{\ell \ell'} \delta_{m m'} \int_0^1 c_{\ell k'} j_{\ell}(\lambda_{\ell k'}r_x) j_\ell(\lambda_{\ell k}r_x)  r_x^2 \dd r_x \\
&= \sum_{(k',\ell',m') \in \mathcal{I}} \frac{\alpha_{k',\ell',m'}}{c_{\ell k}} \delta_{k k'} \delta_{\ell,\ell'} \delta_{m m'} = \frac{\alpha_{k,\ell, m}}{c_{\ell k}},
\end{split}
\end{equation}
for any $(k,\ell, m)$ in $\mathcal{I}$.
\end{proof}

\subsection{Choice of bandlimit} \label{sec:bandlimit}
An important parameter in the algorithm is the choice of bandlimit for the expansion. In this section, we derive a maximum bandlimit based on equating the number of pixels and basis functions.

Let $V = V(h)$ be the number of voxels of size $h \times h \times h$ contained in the unit ball, where $h = 1/\lfloor (N+1)/2\rfloor$ for positive integer $N$. By \cite{heath1999lattice} we have
\begin{equation}
V(h) = \frac{4 \pi}{3} h^{-3} + 
\mathcal{O}\left(h^{-\frac{21}{16}-\varepsilon}\right), \quad \text{as} \quad h \rightarrow 0,
\end{equation}
for any fixed $\varepsilon > 0$.
The number of Dirichlet eigenvalues of the Laplacian on the unit-ball in $\mathbb{R}^3$ that are less than or equal to $\lambda^2$ is by Weyl's law \cite{guo2021note}:
\begin{equation} \label{eigenvalueas}
\mathcal{N}(\lambda) = \frac{2}{9 \pi} \lambda^3 - \frac{1}{4}   \lambda^{2} + o(\lambda^2) \quad \text{as} \quad \lambda \rightarrow \infty. 
\end{equation}
Equating $V(h) = \mathcal{N}(\lambda)$ gives a suitable bandlimit condition since beyond this bandlimit, there will be more basis functions than voxels. 
For simplicity, we equate the leading-order terms of $V(h)$ and $\mathcal{N}(\lambda)$, which gives
\begin{equation} \label{eq:bandlimit_on_lambda}
 \lambda = 6^{1/3 }\pi^{2/3} \frac{1}{h} .
\end{equation}
Since we discarded the second term in
\eqref{eigenvalueas}, this estimate on $\lambda$ is a slight underestimate compared to equating the exact expressions of $V(h)$ and $\mathcal{N}(\lambda)$.
We remark that, in practice, 
  it may be advantageous to compute expansions using a fraction of the bandlimit \eqref{eq:bandlimit_on_lambda}, say, half the maximum bandlimit if noise in the data may make the high-frequency coefficients less informative;
  in any case, choosing $\lambda$ on the order of $h^{-1}$ such that it is less than \eqref{eq:bandlimit_on_lambda}
   is reasonable.

\subsection{Precise setting}
We now make the setting of Theorem~\ref{thm:main} precise. 
We must specify the enumerations \eqref{eq:f-list} and \eqref{eq:psi-list} in the definitions of the operators $B$ and $B^*$.
We fix the enumeration of the $V = N^3$ voxels to be lexicographic, with no loss of generality.  
For the ball harmonics we switch between 
indexing them by the triples $(k,\ell , m)$ and enumerating them by sequential indices
\begin{equation}
\psi_1, \psi_2, \ldots , \psi_n,
\end{equation}
after fixing an ordering of the $(k,\ell,m)$ triples. 
The triples are ordered by sorting the ball harmonics via increasing values of $\lambda_{\ell k}$ with ties broken in the order $m = 0, -1, 1, -2, 2, \ldots , -\ell, \ell $.
Optionally, our implementation lets users specify integers $L$ and $K$, and consider only basis triples with $\ell \leq L$ and $k \leq K$.
The values of the indices $(k,\ell,m)$ corresponding to the sequential ordering are denoted by
\begin{equation}\label{eq:kilimi}
k_i, \ell_i , m_i, \quad \text{ for } \quad i = 1, \ldots , n.
\end{equation}
For convenience, for a given bandlimit $\lambda$ (see Section~\ref{sec:bandlimit} for one choice of $\lambda$) we \nolinebreak let
\begin{equation} \label{eq:list-lambda}
\lambda_1 \leq \ldots \leq \lambda_n ,
\end{equation}
be an enumeration of the $\lambda_{\ell k}$ with magnitude at most $\lambda$, where the $\lambda_{\ell k}$ are repeated according to their multiplicity. 

We will refer to the maximum values of $\ell$ and $k$ that correspond to a $\lambda_{\ell k} \leq \lambda$ by
\begin{align}\label{eq:def_of_L}
    L &= \max \{\ell \in \mathbb{Z}_{\geq 0}: \text{ there exists } k \in \mathbb{Z}_{>0} \text{ with } \lambda_{\ell k} \leq \lambda \}, \\
    K &= \max \{k \in \mathbb{Z}_{>0} : \text{ there exists } \ell \in \mathbb{Z}_{\geq 0} \text{ with } \lambda_{\ell k} \leq \lambda \}.\label{eq:def_of_K}
\end{align}
By \cite[Eq. 1.6]{elbert2001some}, the roots $\lambda_{\ell k}$ satisfy $\lambda_{\ell k} > \ell + 1  +\pi(k-1/2)$.  Therefore $\lambda_{\ell k} > \ell $ and $\lambda_{\ell k} > k$, which imply the (slightly pessimistic) bounds
\begin{equation}\label{eq:bandlimit_on_L}
    L \leq \lambda \quad    \text{and} \quad K \leq \lambda .
\end{equation}

\subsection{Fast algorithms}\label{sec:full_algorithm}
This section leverages the analytical identities of Section~\ref{sec:analytic_identities} to produce the fast algorithms of Theorem~\ref{thm:main}.

The overall idea of the algorithms is to compute the basis expansion coefficients by using \eqref{eq:analytic_beta_equation} and \eqref{eq:final_identity}. However, these identities only apply to continuous functions $f:[-1,1]^3 \rightarrow \mathbb{C}$, whereas Theorem~\ref{thm:main} assumes a discretized input. A main insight of our algorithms is that  \eqref{eq:analytic_beta_equation} and \eqref{eq:final_identity} have discrete counterparts, where replacing the continuous Fourier-transforms by the discrete Fourier-transforms approximately replaces the continuous basis coefficients $\alpha_{k,\ell, m}$ in \eqref{eq:final_identity} by the vector $B^*f$ from \eqref{eq:defB*}. 
It may not be immediately clear that a discrete version of these identities would lead to numerically accurate code.  However, a careful analysis proves that the resulting algorithms are both fast and accurate when using enough discretization points. In Lemmas~\ref{lem:num_radial_nodes} and \ref{lem:num_angular_nodes}, we establish sufficient estimates for the number of discretization points needed.

Crucially, we obtain the computational complexity of Theorem~\ref{thm:main} by not computing \eqref{eq:analytic_beta_equation} and \eqref{eq:final_identity} for \textit{all} $\lambda_{\ell k}$, but rather only for a smaller set of points $\rho_q$.  We then extend the result to all $\lambda_{\ell k}$ via 1D polynomial interpolation. In particular, with $\lambda_1$ and $\lambda_n$ as defined in \eqref{eq:list-lambda}, we take the smaller set of points to be Chebyshev nodes of the first kind on the interval $[\lambda_1, \lambda_n]$, i.e.,
\begin{equation}\label{eq:chebyshev_grid}
    \rho_q = \frac{\lambda_n - \lambda_1}{2} \cos\left(\frac{2q+1}{Q}\cdot \frac{\pi}{2}\right) + \frac{\lambda_1 + \lambda_n}{2}, \quad q = 0, \ldots , Q-1.
\end{equation}
The one-dimensional polynomial interpolating the function values $y_q$ from the source points $\rho_q$ to the target point $x$ is then
\begin{align}\label{eq:def_interp_poly}
P(x) = \sum_{q=0}^{Q-1} y_q u_q(x)  \text{ where } u_q(x) = \frac{\prod_{r \neq q} (x-\rho_r) }{\prod_{r \neq q} (\rho_q-\rho_r)}, \text{ for } q = 0, \ldots , Q-1.
\end{align}

All together, the fast procedure for applying $\widetilde{B}^*$ is as follows (see also Figure~\ref{fig:alg_illustration}):
\begin{enumerate}
\item Use the nonuniform fast Fourier transform (NUFFT) to approximately evaluate $\widehat{f}$ in \eqref{eq:analytic_beta_equation} on a grid with Chebyshev nodes in the radial direction and suitable spherical nodes.
\item Use a fast spherical harmonics transform to approximate the integral in \eqref{eq:analytic_beta_equation} at Chebyshev nodes $\rho$.
\item Use fast interpolation from the Chebyshev nodes $\rho$ to the points $\lambda_{\ell k}$ to approximate \eqref{eq:final_identity} (see Remark~\ref{rem:interpolation_methods} and the references therein for choices of methods for fast interpolation).
\end{enumerate}

\begin{figure}
    \centering
\includegraphics[width=\textwidth]{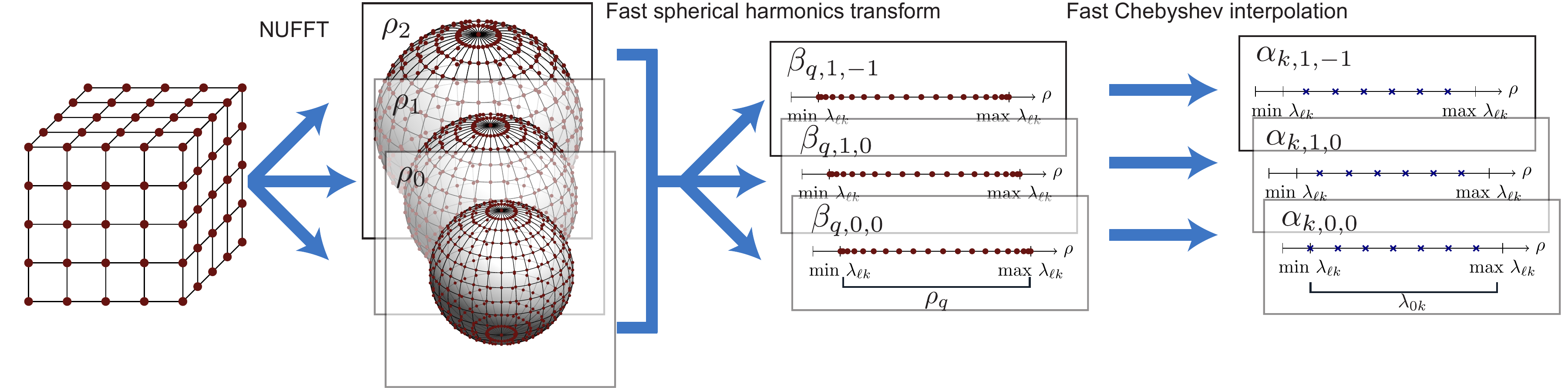}
    \caption{Illustration of the algorithm for applying $B^*$, where the arrows correspond to the steps of the algorithm.}
    \label{fig:alg_illustration}
\end{figure}

Note that each of the three steps above involved in applying the approximation to $B^*$ are linear transforms whose adjoints can be applied with the same low computational complexity.  For example, the adjoint of a type-I NUFFT is a type-II NUFFT.  Thus, approximately applying the operator $B$ can be achieved by applying each of the adjoints in the reverse order, with the same overall computational complexity as the approximate application of $B^*$. We summarize the two algorithms in Algorithms~\ref{algoBH} and \ref{algoB}.

\begin{algorithm}[ht!] 
\caption{Fast application of $\widetilde{B}^*$,}\label{algoBH} 
\setcounter{AlgoLine}{0}
\KwIn{Discretized volume $f \in \mathbb{C}^{V}$, bandlimit $\lambda$, and precision $\varepsilon$.}
\Constants{Maximum number $n$ of basis functions, 
$\varepsilon^\texttt{dis}$ given by \eqref{eq:epsdis} below,
$\varepsilon^\texttt{nuf}$, $\varepsilon^\texttt{fsh}$ and $\varepsilon^\texttt{in}$ given by \eqref{eq:nufst} below,
$
\textstyle
S = \lceil \max\{ 2e 6^{1/3} \pi^{2/3} \lfloor (V^{1/3}+1)/2 \rfloor  , 4|\log_2 (27.6\varepsilon^\texttt{dis})| \} \rceil,\,$ and
$Q = \lceil \max\left\{ 5.3 V^{1/3} , |\log_2 \varepsilon^\texttt{dis}| \right\} \rceil.
$
}
\KwOut{$\widetilde{B}^* f = \alpha \in \mathbb{C}^n$ approximating $B^* f$ to relative error $\varepsilon$ (see Theorem \ref{thm:main}).}
Let $\rho_q$ and $\gamma_{s,t}$ be given by \eqref{eq:chebyshev_grid} and \eqref{eq:angular_grid}, respectively. Using the NUFFT, calculate
$$
a_{qst} = \sum_{j=1}^V f_j e^{-\imath x_j \cdot \rho_{q}\gamma_{s,t}},
    $$
with relative error $\varepsilon^\texttt{nuf}$, where $q \in \{0, \ldots , Q-1\}$, $s \in \{ 0, \ldots, S\} $ and $t \in \{0, \ldots , S-1\}$.

Using a fast spherical harmonics transform, calculate
$$
\beta_{q, \ell, m} = \frac{ \imath^\ell}{4\pi} \sum_{s=0}^S\sum_{t=0}^{S-1} w_s a_{q s t}  \overline{Y_\ell^m(\gamma_{s,t})},
$$
for $q \in \{  0, \ldots , Q-1\}$, $\ell \in \{ 0, \ldots , L\}$ and $m \in  \{-\ell, \ldots , \ell\}$, where the $w_s$ are quadrature weights given by \eqref{eq:clenshaw_curtis_weights}.

For $\ell = 0, \ldots , L$, and $m = -\ell, \ldots , \ell$, use fast Chebyshev interpolation to calculate
$$
\alpha_i = 
\sum_{q=0}^{Q-1} c_i \beta_{q \ell_i m_i} u_q(\lambda_i)  h^{3/2},
$$
where $u_q(\cdot)$ are the interpolating polynomials given by \eqref{eq:def_interp_poly}.
\end{algorithm}

\begin{algorithm}[ht!] 
\caption{Fast application of $\widetilde{B}$.}\label{algoB} 
\setcounter{AlgoLine}{0}
\KwIn{Coefficients $\alpha \in \mathbb{C}^n$, bandlimit $\lambda$, and precision $\varepsilon$.}
\Constants{Number of voxels $V$, $\varepsilon^\texttt{dis}$ given by \eqref{eq:epsdis2},
$\varepsilon^\texttt{in}$,$\varepsilon^\texttt{fsh}$ and
$\varepsilon^\texttt{nuf}$ given by \eqref{eq:error_nufft_B_proof},
$
\textstyle
S = \lceil \max\{ 2e 6^{1/3} \pi^{2/3} \lfloor (V^{1/3}+1)/2 \rfloor  , 4|\log_2 (27.6\varepsilon^\texttt{dis})| \} \rceil,\,$ and
$Q = \lceil \max\left\{ 5.3 V^{1/3} , |\log_2 \varepsilon^\texttt{dis}| \right\} \rceil.
$}

\KwOut{$f \in \mathbb{C}^V$ approximating $B \alpha$ to relative error $\varepsilon$ (see Theorem \ref{thm:main})}
For $q= 0, \ldots , Q-1$, $\ell = 0, \ldots , L$, and $m = -\ell, \ldots , \ell$, use the adjoint of fast Chebyshev interpolation to apply
$$
\beta^*_{q, \ell, m} =
\sum_{\substack{i: \ell_i = \ell, \\ \quad m_i = m}} c_i 
\alpha_{i}
 u_q(\lambda_i)  h^{3/2},
$$
where the sum is over $i$ such that the corresponding sequential indices from \eqref{eq:kilimi} satisfy $\ell_i = \ell$ and $m_i = m$.

Using a fast spherical harmonics transform, calculate
$$
  a_{q s t}^* 
    = \sum_{\ell=0}^
    {L} \frac{ (-\imath)^\ell}{4\pi}
    \sum_{m =-\ell}^{\ell} w_s 
    \beta_{q, \ell, m}^* 
   Y_\ell^m(\gamma_{s,t}),
$$
where $q \in \{  0, \ldots , Q-1\}$, $s \in \{0, \ldots, S\}$, $t \in \{ 0, \ldots , S-1\}$ and the $w_s$ are quadrature weights given by \eqref{eq:clenshaw_curtis_weights}.

With $\rho_q$ given in \eqref{eq:chebyshev_grid} and $\gamma_{s,t}$ in \eqref{eq:angular_grid},
use the NUFFT to calculate the quantities
$$
    f_j = \sum_{q=0}^{Q-1}\sum_{s=0}^S \sum_{t=0}^{S-1} a_{q s t}^* e^{\imath x_j \cdot \rho_{q}\gamma_{s,t}},
$$
for each $j = 1, \ldots , V$.
\end{algorithm}

\section{Key discrete results}\label{sec:discrete_results}
The most important part of both the proof of Theorem~\ref{thm:main} and the practical implementation of Algorithms~\ref{algoBH} and \ref{algoB} is ensuring that the scaling of the number of discretization nodes is correct. In this section, we derive a bound on the number of nodes that achieves a desired level of precision with the complexity written in Theorem~\ref{thm:main}.

\subsection{Number of radial nodes}
Our first lemma states the number of radial nodes required to obtain a prescribed accuracy with the desired complexity.
\begin{lemma} \label{lem:num_radial_nodes}
Let the number of radial Chebyshev nodes in \eqref{eq:chebyshev_grid} be
\begin{equation}
Q = \left\lceil \max\{ 5.3 V^{1/3} , \log_2 \eta^{-1}\}  \right\rceil.
\label{eq:num_radial_nodes}
\end{equation}
Write $x_j = r_{x_j}\gamma_{x_j}$ as in \eqref{eq:def_spherical_gamma}, with $0 \leq r_{x_j} \leq 1$. Let $P_{\ell,m,j}$ be the degree $Q-1$ interpolating polynomial in \eqref{eq:def_interp_poly} satisfying
$$
P_{\ell,m,j}(\rho_q) =  j_\ell(r_{x_j} \rho_q)\overline{Y^m_\ell(\gamma_{x_j})},
$$
for $q \in \{0,\ldots,Q-1\}$, 
where $\rho_q$ are the Chebyshev nodes for $[\lambda_1,\lambda_n]$ defined in \eqref{eq:chebyshev_grid} and $\lambda_n$ satisfies \eqref{eq:bandlimit_on_lambda}. Then,
$$
|P_{\ell,m,j}(\rho_k) -  j_\ell(r_{x_j} \rho_k)\overline{Y^m_\ell(\gamma_{x_j})}| \le \eta,
$$ 
for all $\rho \in [\lambda_1,\lambda_n]$, $\ell \in \{0,\ldots,L\}$, $m\in \{-\ell, \ldots, \ell\}$, and $j \in \{1,\ldots,V\}$, where $L$ is defined in \eqref{eq:def_of_L}.
\end{lemma}

\begin{proof}[Proof of Lemma \ref{lem:num_radial_nodes}] 
Let $h : [a,b] \rightarrow \mathbb{R}$ be a smooth function, and $P$ be the polynomial that interpolates $h$ at $Q$ Chebyshev nodes.  Then, the residual $R(\rho) = h(\rho) - P(\rho)$  satisfies
\begin{equation*}
    |R(\rho)| \leq \frac{C_{Q}}{Q!} \left(\frac{b-a}{4}\right)^{Q}, 
    \quad \text{where} \,\, C_Q := \max_{\rho \in [a,b]} |h^{(Q)}(\rho) |
\end{equation*}
for all $\rho \in [a,b]$; see \cite[Lemma~2.1]{rokhlin1988fast}.
Set
\begin{equation*}
[a,b] := [\lambda_{1}, \lambda_{m}], \quad \text{and} \quad
h(\rho) := j_\ell(r_{x_j} \rho)\overline{Y^m_\ell(\gamma_{x_j})}. 
\end{equation*}
First, we estimate $|R(\rho)|$ by
\begin{equation*}
    |R(\rho)| \leq \frac{C_{Q}}{Q !} \left(\frac{\lambda_{m} - \lambda_{1}}{4}\right)^{Q} \leq \frac{C_{Q} }{Q !}
    \left(
\left(\frac{\sqrt{3}\pi}{16}\right) ^{2/3} (V^{1/3}+1)
    \right)^Q,
\end{equation*}
where we used the fact that
$$
\lambda_n-\lambda_1 \le \lambda_n \leq   \left(\frac{\sqrt{3}\pi}{2}\right) ^{2/3}  (V^{1/3}+1); $$ 
see \S \ref{sec:bandlimit} for the last inequality. 
Next, we estimate 
$$
C_{Q} := \max_{\rho \in [\lambda_1,\lambda_n]} \left| \frac{\partial ^ {Q}}{\partial \rho^{Q}}\left(j_\ell(r_{x_j} \rho)\right)\overline{Y^m_\ell(\gamma_{x_j})} \right|,
$$
by expanding the function $j_\ell(r_{x_j} \rho)$ using the integral identity in \cite[Eq. 10.54.2]{dlmf} and obtain
\begin{eqnarray*}
\left|\frac{\partial ^ Q}{\partial \rho^Q} j_\ell(r_{x_j} \rho) \right| &=& 
\left|\frac{1}{2} \int_0^{\pi}\frac{\partial ^ Q}{\partial \rho^Q} \left( e^{\imath r_{x_j}\rho \cos(\theta)} P_\ell(\cos \theta) \sin \theta \right)\dd \theta \right| \\
&=&  \left| \frac{1}{2} \int_0^{\pi} \left(\imath r_{x_j} \cos(\theta) \right)^Q e^{\imath r_{x_j}\rho \cos(\theta)} P_\ell(\cos \theta) \sin \theta\dd \theta \right| \\
&\leq&  \frac{1}{2} \int_0^{\pi}  |P_\ell(\cos \theta)| \sin \theta \text{d}\theta = \frac{1}{2} \int_{-1}^{1}  |P_\ell(x)| \dd x \\
&\leq& \frac{1}{2} \sqrt{\frac{2}{2\ell + 1}} \sqrt{2} = \frac{1}{\sqrt{2\ell + 1}},
\end{eqnarray*}
where the first inequality follows from the triangle inequality and the fact that $0 \leq r_{x_j} \leq 1$. The second inequality follows from Cauchy-Schwarz and the fact that  $\sqrt{2/(2\ell+1)}$ is the $L^2$-norm of $P_\ell$ on the interval $[-1,1]$, see \cite[Table 18.3.1]{dlmf}. Combining this with the bound $|Y_\ell^m(\gamma_{x_j})| \leq \sqrt{\frac{2\ell + 1}{4\pi}}$ from \eqref{eq:unsold_bound}, we obtain 
$$
C_Q \leq \frac{1}{\sqrt{4\pi}}.
$$
In combination with Stirling's approximation \cite{robbins1955remark}, it follows that
\begin{equation}\label{eq:optimize_for_Q}
\begin{split}
|R(\rho)| &\le \frac{1}{\sqrt{4\pi} Q!} \left( \left(\frac{\sqrt{3}\pi}{16}\right) ^\frac{2}{3} (V^\frac{1}{3}+1)
    \right)^Q \\
    &\leq  
     \frac{1}{\sqrt{8Q}\pi} \left( \frac{e}{Q}\left(\frac{\sqrt{3}\pi}{16}\right) ^\frac{2}{3} (V^\frac{1}{3}+1)
    \right)^Q. 
    \end{split}
\end{equation}
We note that $V^{1/3} \geq 1$, $2e\left(\frac{\sqrt{3}\pi}{16}\right) ^{2/3} \leq 2.65$ and $\frac{1}{\sqrt{8Q}\pi} \leq 1$.
Therefore, in order to achieve  error $|R(\rho)| \leq \eta$, it suffices to set $Q$ such that
\begin{equation} \label{eq:nonlinear_q}
 \eta \geq \left( \frac{ 2.65 V^{1/3}}{Q} \right)^Q.
\end{equation}
Setting $\frac{ 2.65 V^{1/3}}{Q} = 1/2$
and solving for $Q$ gives
$
Q = 5.3 V^{1/3}
$. This implies that choosing
$
Q \ge \max\{ 5.3 V^{1/3} , \log_2 \eta^{-1}\}
$
is sufficient to achieve  error less than $\eta$. 
\end{proof}

\begin{remark}(Optimizing the bound for $Q$)
The bounds used in the proof of Lemma~\ref{lem:num_radial_nodes} to show the asymptotic behavior of $Q$ are slightly conservative. In practice, we can numerically obtain smaller values of $Q$ through \eqref{eq:optimize_for_Q}, by choosing the smallest integer $Q$ satisfying
$$
\frac{1}{\sqrt{4\pi} Q!} \left( \left(\frac{\sqrt{3}\pi}{16}\right) ^{2/3} (V^{1/3}+1)
    \right)^Q \leq \eta.
$$
\end{remark}

\subsection{Number of spherical nodes}
Our second lemma similarly gives the number of spherical nodes needed to obtain a prescribed accuracy with the desired complexity.

\begin{lemma} \label{lem:num_angular_nodes}
Let $\eta > 0$ be given.
With $\gamma_{s,t}, w_s$ from \eqref{eq:angular_grid} and \eqref{eq:clenshaw_curtis_weights}, $L$ from \eqref{eq:def_of_L} and $\lambda_1, \lambda_n$ from \eqref{eq:list-lambda}, 
let $\rho \in [\lambda_1,\lambda_n]$ and define
$$
R_{\ell, m, j}(\rho) = \left| \frac{ \imath^\ell}{4\pi} \sum_{s = 0}^{S}\sum_{t = 0}^{S-1} w_s  \overline{Y_\ell^m(\gamma_{s,t})}e^{-\imath x_j \cdot \rho\gamma_{s,t}} - j_\ell(r_{x_j} \rho)\overline{Y^m_\ell(\gamma_{x_j})} \right|,
$$
 where $x_j = r_{x_j} \gamma_{x_j}$ is written in the form in \eqref{eq:def_spherical_gamma} with $0 \leq r_{x_j} \leq 1$.
If $\lambda_n$ satisfies \eqref{eq:bandlimit_on_lambda} and $S$ satisfies
\begin{equation}
\label{eq:num_angular_nodes_phi}
 \textstyle
S \geq \lceil \max\{ 2e 6^{1/3} \pi^{2/3} \lfloor (V^{1/3}+1)/2 \rfloor , 4\log_2(27.6\eta^{-1}) \} \rceil,
\end{equation}
then
$
R_{\ell, m, j}(\rho) \leq \eta,
$
for $\ell \in \{0,\ldots,L\}$, $m \in \{-\ell, \ldots , \ell\}$, and $j \in \{1,\ldots,V\}$.
\end{lemma}
\begin{proof}[Proof of Lemma \ref{lem:num_angular_nodes}]
Assume $S$ satisfies \eqref{eq:num_angular_nodes_phi}.
Let
\begin{equation} \label{eq:g_def}
g_{\ell, m, j}(\rho,\gamma) = \frac{ \imath^\ell}{4\pi}  \overline{Y_\ell^m(\gamma)}e^{-\imath x_j \cdot \rho\gamma}.
\end{equation}
We want to show that
\begin{equation} \label{eq:sum_ang_g}
R_{\ell, m, j}(\rho) = \left|  \sum_{s = 0}^{S}\sum_{t = 0}^{S-1} w_s g_{\ell, m, j}(\rho,\gamma_{s,t}) - j_\ell(r_{x_j} \rho)\overline{Y^m_\ell(\gamma_{x_j})}  \right| < \eta.
\end{equation}
Notice that the sum in \eqref{eq:sum_ang_g}  is a discretization of the integral
\begin{equation}\label{eq:exact_integral}
\int_{\mathbb{S}^2} g_{\ell, m, j}(\rho, \gamma) \dd\sigma(\gamma) = j_\ell(r_{x_j} \rho)\overline{Y^m_\ell(\gamma_{x_j})}
,
\end{equation}
where this exact expression for the integral follows from \eqref{eq:plane_wave_expansion}.
For fixed $j$ and $\rho$, we define the constants
\begin{equation}\label{def:alpha_in_lemma_bound}
\alpha_{m', \ell'} =  4\pi(-\imath)^{\ell'} j_{\ell'}(r_{x_j} \rho) \overline{Y_{\ell'}^{m'}(\gamma_{x_j})},
\end{equation}
and use the plane-wave expansion in \eqref{eq:plane_wave_expansion} to write
\begin{equation}
e^{-\imath \rho \gamma \cdot x_j} =   \sum_{\ell'=0}^\infty  \sum_{m'=-\ell'}^{\ell'} \alpha_{m', \ell'} Y_{\ell'}^{m'}(\gamma)  .
\end{equation}
By \eqref{eq:num_angular_nodes_phi} and \eqref{eq:bandlimit_on_L}, we have $L \le S/2$.
Recall from \eqref{clenshaw-curtis-equality} that the quadrature rule defined by \eqref{eq:angular_grid} and \eqref{eq:clenshaw_curtis_weights} is exact for integrals corresponding to inner products of bandlimited functions on the sphere when $L \le S/2$. This gives
\begin{align}
R_{\ell, m, j}(\rho)&= \left| \sum_{s = 0}^{S}\sum_{t = 0}^{S-1} w_s\frac{\imath^\ell}{4\pi} e^{-\imath x_j \cdot \rho \gamma_s }\overline{Y_\ell^m(\gamma_{s,t})} - \! \int_{\mathbb{S}^2} \frac{\imath^\ell}{4\pi}e^{-\imath x_j \cdot \rho \gamma_\omega }\overline{Y_\ell^m(\gamma_\omega)} \dd\sigma(\gamma_\omega) \right| \nonumber \\
&= \left| \sum_{s = 0}^{S}\sum_{t = 0}^{S-1} w_s\frac{\imath^\ell}{4\pi} \sum_{\ell'=0}^\infty  \sum_{m'=-\ell'}^{\ell'} \alpha_{m', \ell'} Y_{\ell'}^{m'}(\gamma_{s,t}) \overline{Y_\ell^m(\gamma_{s,t})} \right. \nonumber \\
& \qquad \quad \left. - \int_{\mathbb{S}^2}  \frac{\imath^\ell}{4\pi}\sum_{\ell'=0}^\infty  \sum_{m'=-\ell'}^{\ell'} \alpha_{m', \ell'} Y_{\ell'}^{m'}(\gamma_\omega) \overline{Y_\ell^m(\gamma_\omega)} \dd\sigma(\gamma_\omega) \right| \label{eq:def_Rlmj} \\
&= \left| \sum_{s = 0}^{S}\sum_{t = 0}^{S-1} \left( \sum_{\ell' = \lfloor S/2 \rfloor + 1}^{ \infty} \sum_{m'=-\ell'}^{\ell'} \frac{\imath^\ell}{4\pi}w_s\alpha_{m',\ell'} Y_{\ell'}^{m'}(\gamma_{s,t})\overline{Y_\ell^m(\gamma_{s,t})}\right) \right. \nonumber \\
&\left. \qquad \quad - \sum_{\ell' = \lfloor S/2 \rfloor + 1}^{ \infty} \sum_{m'=-\ell'}^{\ell'} \frac{\imath^\ell}{4\pi} \alpha_{m',\ell'} \int_{\mathbb{S}^2} 
 Y_{\ell'}^{m'}(\gamma_\omega) \overline{Y_\ell^m(\gamma_\omega)}d\sigma(\gamma_\omega) \right|. \nonumber
\end{align}
We proceed by bounding the coefficients $\alpha_{m',\ell'}$ in order to show that the preceding sum is bounded in magnitude by $\eta$. By the definition of $\alpha_{m', \ell'}$ in \eqref{def:alpha_in_lemma_bound}
\begin{align}
\left|\frac{\imath^\ell}{4\pi}\alpha_{m', \ell'} \right| &=    \left| j_{\ell'}(r_{x_j} \rho) \overline{Y_{\ell'}^{m'}(\gamma_{x_j})} \right| \nonumber \leq \frac{\sqrt{\pi}}{2}\frac{\left(\frac{r_{x_j}\rho}{2}\right)^{\ell'}}{\Gamma(\ell'+3/2)} \sqrt{\frac{2\ell'+1}{4\pi}} \nonumber\\[2ex]
    &\leq \frac{\sqrt{\pi}}{2}\sqrt{\frac{\ell'+3/2}{2\pi}}\left(\frac{e}{\ell' + 3/2}\right)^{\ell'+3/2}\left( \frac{r_{x_j}\rho}{2}\right)^{\ell'}\sqrt{\frac{2\ell'+1}{4\pi}} \label{eq:bound_alpha} \\
    &= \frac{e^{3/2}}{4\sqrt{2\pi}} \frac{\sqrt{2\ell'+1}}{\ell'+3/2}\left( \frac{er_{x_j}\rho}{2\ell' + 3}\right)^{\ell'} \leq \frac{1}{3}\left( \frac{er_{x_j}\rho}{2\ell' + 3}\right)^{\ell'},\nonumber
\end{align}
where the first inequality follows from \eqref{jnJN}, \cite[10.14.4]{dlmf} and \eqref{eq:unsold_bound}, the second inequality follows from \cite[5.6.1]{dlmf} and the final equality follows from the fact that $$\frac{e^{3/2}}{4\sqrt{2\pi}} \frac{\sqrt{2\ell'+1}}{\ell'+3/2} \leq \frac{1}{3},$$ for non-negative integers $\ell'$, with $\ell' = 1$ maximizing the left hand side.
Using $|w_s| \leq \frac{8\pi}{S^2}$ from \eqref{eq:bound_ws}, it follows that 
\begin{align*}
     &\left| \sum_{s = 0}^{S}\sum_{t = 0}^{S-1} w_s Y_{\ell'}^{m'}(\gamma_s) \overline{Y_\ell^m}(\gamma_s) -    \int_{\mathbb{S}^2} 
 Y_{\ell'}^{m'}(\gamma_\omega)\overline{Y_\ell^m}(\gamma_\omega) \dd\sigma(\gamma_\omega) \right|  \\
 &\qquad \leq \sum_{s = 0}^{S}\sum_{t = 0}^{S-1} \frac{8\pi}{S^2} |Y_{\ell'}^{m'}(\gamma_s) \overline{Y_\ell^m}(\gamma_s)| +    \int_{\mathbb{S}^2}
 |Y_{\ell'}^{m'}(\gamma_\omega)\overline{Y_\ell^m}(\gamma_\omega)| \dd\sigma(\gamma_\omega) \\
 & \qquad \leq \frac{8\pi}{S^2}S(S+1) \sqrt{\frac{2\ell'+1}{4\pi}} \sqrt{\frac{2\ell+1}{4\pi}} + 4\pi \sqrt{\frac{2\ell'+1}{4\pi}} \sqrt{\frac{2\ell+1}{4\pi}} \\
 & \qquad \leq \frac{112\pi}{9} \sqrt{\frac{2\ell'+1}{4\pi}} \sqrt{\frac{2\ell+1}{4\pi}},
\end{align*}
where the second inequality used \eqref{eq:unsold_bound} and the third that $\frac{S(S+1)}{S^2} \leq \frac{19}{18}$ since $S \geq \lceil 2e \lambda_1 \rceil = \lceil 2e \lambda_{01} \rceil = 18$. Inserting this into \eqref{eq:def_Rlmj} and using \eqref{eq:bound_alpha} results in
\begin{align}
R_{\ell, m, j}(\rho) &= \left| \sum_{\ell' = \lfloor S/2 \rfloor+1}^{ \infty} \sum_{m'=-\ell'}^{\ell'}   \frac{\imath^\ell}{4\pi}\alpha_{m',\ell'}\left( \sum_{s = 0}^{S}\sum_{t = 0}^{S-1} w_s Y_{\ell'}^{m'}(\gamma_s) \overline{Y_\ell^m}(\gamma_s) \right. \right. \nonumber \\ & \qquad \qquad \left. \left. - \int_{\mathbb{S}^2} 
 Y_{\ell'}^{m'}(\gamma_\omega)\overline{Y_\ell^m}(\gamma_\omega) \dd\sigma(\gamma_\omega) \right) \right| \nonumber \\
 &\leq \sum_{\ell' = \lfloor S/2 \rfloor+1}^{ \infty} \sum_{m'=-\ell'}^{\ell'}   \left| \frac{\imath^\ell}{4\pi}\alpha_{m',\ell'} \right| \frac{112\pi}{9}  \sqrt{\frac{2\ell'+1}{4\pi}} \sqrt{\frac{2\ell+1}{4\pi}} \nonumber \\
  &\leq \sum_{\ell' = \lfloor S/2 \rfloor+1}^{ \infty} \sum_{m'=-\ell'}^{\ell'}    \frac{112\pi}{27}  \left( \frac{er_{x_j}\rho}{2\ell'+3}\right)^{\ell'} \sqrt{\frac{2\ell'+1}{4\pi}} \sqrt{\frac{2\ell+1}{4\pi}} \nonumber  \\
 & = \frac{28}{27} \sqrt{2\ell +1} \sum_{\ell' = \lfloor S/2 \rfloor+1}^{ \infty} (2\ell'+1)^{3/2}   \left( \frac{er_{x_j}\rho}{2\ell'+3}\right)^{\ell'}\nonumber \\
  &  = \frac{28}{27} \sqrt{2\ell +1} (er_{x_j}\rho)^{3/2} \sum_{\ell' = \lfloor S/2 \rfloor+1}^{ \infty} \left( \frac{er_{x_j}\rho}{2\ell'+3}\right)^{\ell'-3/2}\nonumber \\
   &  \leq \frac{28}{27} \sqrt{2\ell +1} (er_{x_j}\rho)^{3/2} \sum_{\ell' = \lfloor S/2 \rfloor+1}^{ \infty} \left( \frac{er_{x_j}\rho}{2(\lfloor S/2 \rfloor+1)+3}\right)^{\ell'-3/2}. \label{eq:optimize_for_S}
   \end{align}
   Since we assume $S  \geq 2er_{x_j}\rho$, it follows that $\frac{er_{x_j}\rho}{2(\lfloor S/2 \rfloor+1)+3} \leq 1/2$, so
\begin{align*}
      R_{\ell, m, j} &  \leq \frac{28}{27} \sqrt{2\ell +1} (er_{x_j}\rho)^{3/2} \sum_{\ell' = \lfloor S/2 \rfloor+1}^{ \infty} \left( \frac{1}{2}\right)^{\ell'-3/2} \nonumber \\
       &  = \frac{28}{27} \sqrt{2\ell +1} (er_{x_j}\rho)^{3/2} \left( \frac{1}{2}\right)^{{\lfloor S/2\rfloor-3/2}} \\
       & \leq \frac{28}{27} \sqrt{S +1} \left(\frac{S}{2}\right)^{3/2} \left( \frac{1}{2}\right)^{{\lfloor S/2\rfloor-3/2}} \leq 27.6\cdot 2^{-S/4} \nonumber,
\end{align*}
since $$\frac{28}{27} \sqrt{S +1} \left(\frac{S}{2}\right)^{3/2} \left( \frac{1}{2}\right)^{{\lfloor S/2\rfloor-3/2}}2^{S/4} \leq 27.6,$$ for integer $S$, where the left hand side is maximized at $S=11$. Note that $r_{x_j} \leq 1$ and $\rho \leq \lambda_n$, and $L \leq \lambda_n$ by \eqref{eq:bandlimit_on_L}. Using \eqref{eq:bandlimit_on_lambda}, the conclusion then follows.
\end{proof}

\begin{remark}[Comparison to 2D angular node proof] 
The proof strategy for 
Lemma \ref{lem:num_angular_nodes} is more significantly complicated than
the analogous 2D result \cite[Lemma 4.2]{marshall2023fast}. When generalizing the 2D result, we noticed a gap in the 2D proof: the derivative estimate in the equation before \cite[Eq. (4.2)]{marshall2023fast} is not rigorously proven. In the present paper, we use a more complicated strategy, which avoids this issue.
We remark that a similar approach can be used to rigorously establish \cite[Lemma 4.2]{marshall2023fast} up to a constant factor of $2$.
\end{remark}

\begin{remark}[Optimizing the bound for $S$]
The bounds used in the proof of Lemma~\ref{lem:num_angular_nodes} to show the asymptotic behavior of $S$ are slightly conservative. In practice, we can numerically obtain smaller values of $S$ through \eqref{eq:optimize_for_S}, by choosing the smallest integer $S$ satisfying
$$
\frac{28}{27} \sqrt{2\ell +1} (e \lambda_n)^{3/2} \sum_{\ell' = \lfloor S/2 \rfloor+1}^{ \infty} \left( \frac{e \lambda_n}{2(\lfloor S/2 \rfloor+1)+3}\right)^{\ell'-3/2} \leq \eta .
$$
\end{remark}

\section{Numerical experiments}\label{sec:numerical_experiments}
This section applies Algorithms~\ref{algoBH} and \ref{algoB} to benchmarking problems. We first make two remarks on implementation details.

\begin{remark} (Barycentric interpolation) \label{rem:interpolation_methods}
There are many fast methods for performing the interpolation steps of Algorithms~\ref{algoBH} and \ref{algoB} \cite{Dutt1993,Greengard2004,lee2005type,dutt1996,gimbutas2020fast,plonka2018numerical}. In practice, using precomputed sparse barycentric interpolation matrices \cite{berrut2004barycentric} is simple, fast and accurate, and is therefore the method used in our implementation.
\end{remark}
\begin{remark} (FFT-based bandlimit heuristic) \label{rmkbandlimith}
Previously, in \S\ref{sec:bandlimit}, we derived the following upper bound on the bandlimit parameter: 
\begin{equation} \label{bandlimitupperboundest}
\lambda = 6^{1/3} \pi^{2/3} \lfloor (N+1)/2 \rfloor \approx 1.80 N.
\end{equation}
The main results of this paper guarantee that the operators $B$ and $B^*$ can be applied accurately and rapidly for any bandlimit bounded by $\lambda$.
However, for applications that compute with the basis expansion operators $B$ and $B^*$ (e.g., in fast computations of least-squares problems involving these operators), using a slightly smaller bandlimit is necessary since the minimum singular values of $B$ and $B^*$ are very small at this value of $\lambda$. Determining the optimal practical bandlimit is a separate theoretical issue from fast computation and one not addressed here.

A heuristic that we find to be effective is based on the FFT. The maximum frequency of a centered FFT of length $N$ is $\pi^2 (N/2)^2$. Equating this frequency with the eigenvalue $\lambda^2$ gives  
\begin{equation} \label{approxlambdan}
\lambda = \pi/2 N \approx 1.57 N.
\end{equation}
Since this heuristic leads to a bandlimit less than
\eqref{bandlimitupperboundest}, the accuracy guarantees of Theorem~\ref{thm:main} apply.  We use the bandlimit \eqref{approxlambdan} for the numerical experiments below.

\begin{figure}[htb]
\centering
\begin{tabular}{cc}
\includegraphics[width=.465\textwidth]{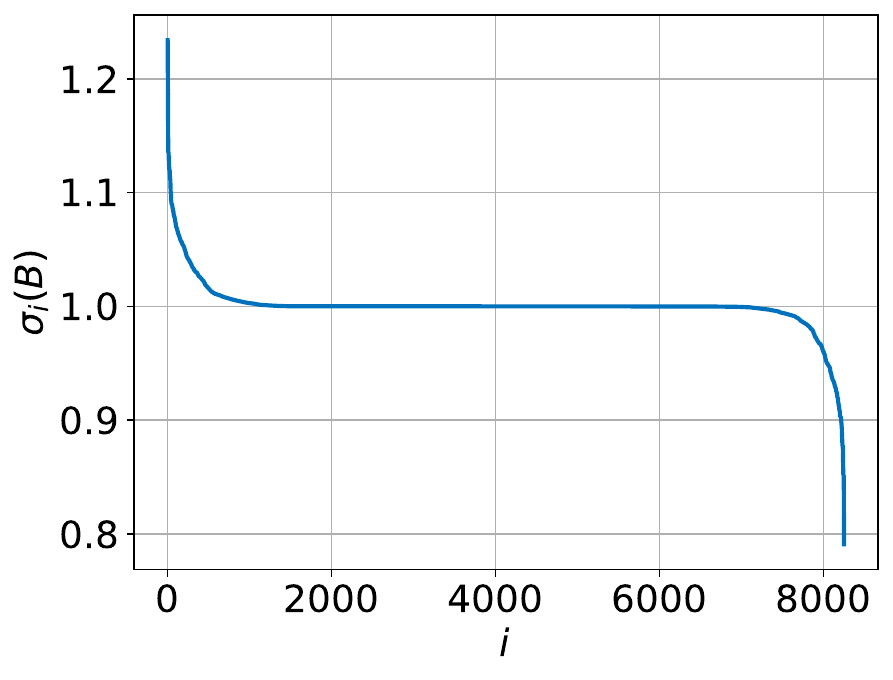} &
\includegraphics[width=.465\textwidth]{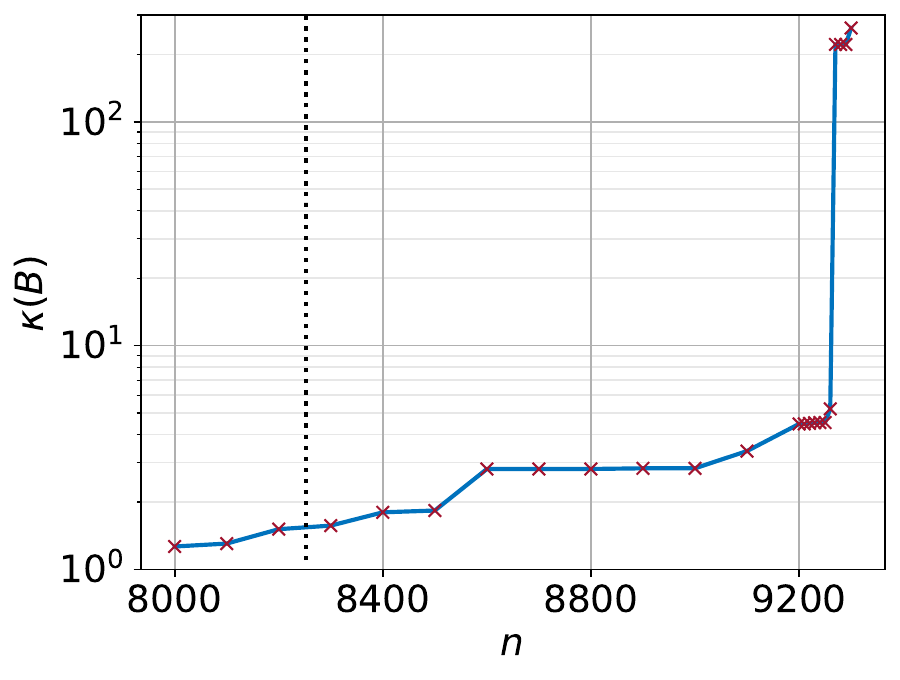}
\end{tabular}
\caption{ Singular values of $B$ when using the bandlimit $\lambda$ from \eqref{approxlambdan} (left) and
condition number $\kappa(B)$ as a function of the number of basis functions $n$ (right). We mark the number of basis functions resulting from using the bandlimit from \eqref{approxlambdan} as a dashed vertical line.}
\label{fig:condb}
\end{figure}

To illustrate the conditioning of $B$, Figure~\ref{fig:condb} plots the condition number of $B$ for volumes with $V = 32 \times 32 \times 32$ voxels as a function of the number of basis functions $n$ used. We mark the number of basis functions resulting from using the bandlimit from \eqref{approxlambdan} as a dashed vertical line. We also plot the singular values of $B$ when using the bandlimit \eqref{approxlambdan}, showing that the $B$ is well-conditioned with this choice of bandlimit.

\end{remark}

\subsection{Accuracy} \label{sec:accuracy}
This section verifies the accuracy guarantees of Theorem~\ref{thm:main}. We compare the fast methods of Algorithms~\ref{algoBH} and \ref{algoB} to the dense method performed by explicitly forming and applying the matrix $B$ from \eqref{eq:defB} with entries $B_{ij} = \psi_i(x_j)h^{3/2}$ and its adjoint $B^*$ from \eqref{eq:defB*}, respectively. We denote the relative $\ell^1$ - $\ell^\infty$ error comparing applications of $B$ and $B^*$ to their fast counterparts of Theorem~\ref{thm:main} by
\begin{equation}
    \text{err}_\alpha = \frac{\|B\alpha - \widetilde{B}\alpha\|_{\ell^\infty}}{\|\alpha\|_{\ell^1}} \quad \text{ and } \quad  \text{err}_f = \frac{\|B^* f - \widetilde{B}^* f\|_{\ell^\infty}}{\| f\|_{\ell^1}} .
\end{equation}

We verify that the Algorithms~\ref{algoBH} and \ref{algoB} achieve the $\ell^1$ - $\ell^\infty$ error guaranteed by Theorem~\ref{thm:main}. 
The experiments use as volume a 3D density map of the SARS-CoV-2 Omicron spike glycoprotein complex  \cite{guo2022structures} downloaded from the online electron microscopy data bank~\cite{lawson2016emdatabank} and illustrated in Figure~\ref{fig:downsampling}. The computational results are shown in Table~\ref{table:accuracy}. Explicitly forming the dense matrices $B$ becomes computationally prohibitive for moderate values of $N$ and we can therefore only compute the dense operators up to $N=56$. The experiments were performed on a machine with Intel Xeon E7-4870 processors and 1.5 TB of memory.

As an additional validation of the usefulness of Theorem~\ref{thm:main}, we also report the $\ell^2$ - $\ell^2$ errors
\begin{equation}
    \text{err}_{2,\alpha} = \frac{\|B\alpha - \widetilde{B}\alpha\|_{\ell^2}}{\|B\alpha\|_{\ell^2}} \quad \text{ and } \quad  \text{err}_{2,f} = \frac{\|B^* f - \widetilde{B}^* f\|_{\ell^2}}{\| B^*f\|_{\ell^2}} .
\end{equation}

\begin{table}
    \centering
        \caption{Relative accuracy of Algorithms~\ref{algoBH} and \ref{algoB}.}
    \label{table:accuracy}
\begin{tabular}{r||cccccc}
\hline
$N$ & $\varepsilon$ & $\text{err}_\alpha$ & $\text{err}_f$& $\text{err}_{2,\alpha}$ & $\text{err}_{2,f}$ \\
\hline
32 &  1.00000e-04 &  1.09147e-06 &  3.44431e-07  &  3.47059e-04 &  3.94597e-05 \\
32 &  1.00000e-07 &  8.80468e-10 &  7.31137e-10  &  3.77298e-07 &  7.32120e-08 \\
32 &  1.00000e-10 &  1.50301e-15 &  9.66415e-16  &  3.90013e-13 &  8.60492e-14\\
32 &  1.00000e-14 &  9.21641e-17 &  1.61313e-16  &  2.79836e-14 &  1.44868e-14 \\
\hline
\hline
48 &  1.00000e-04 &  3.07614e-07 &  1.43754e-07  &  3.10650e-04 &  3.69283e-05 \\
48 &  1.00000e-07 &  3.93735e-09 &  9.17106e-10  &  2.42870e-06 &  1.64956e-07 \\
48 &  1.00000e-10 &  1.71840e-13 &  2.63019e-14  &  1.22662e-10 &  5.45878e-12\\
48 &  1.00000e-14 &  8.80011e-15 &  3.44285e-15  &  5.11352e-12 &  6.40179e-13 \\
\hline
\hline
56 &  1.00000e-04 &  1.46292e-07 &  1.00198e-07  &  2.68961e-04 &  4.08369e-05\\
56 &  1.00000e-07 &  1.80844e-09 &  5.80538e-10  &  2.90177e-06 &  2.10560e-07 \\
56 &  1.00000e-10 &  2.87442e-13 &  6.20621e-14  &  3.34660e-10 &  1.56999e-11 \\
56 &  1.00000e-14 &  5.10866e-14 &  7.97114e-15  &  5.65145e-11 &  2.36787e-12\\
\hline
\end{tabular}
\end{table}

The  $\ell^1$ - $\ell^\infty$ errors are bounded by $\varepsilon$, as guaranteed by Theorem~\ref{thm:main} and the $\ell^2$ - $\ell^2$ errors are close to $\varepsilon$, even though this is not a priori guaranteed by Theorem~\ref{thm:main} given the hyperparameter settings chosen for the experiments.

\subsection{Timings}

This section verifies the computational complexity given in Theorem~\ref{thm:main}. We record the timings of the fast algorithms with two different choices of fast spherical harmonics transforms \cite{bonev2023spherical,slevinsky2019fast}, which have complexity $\mathcal{O}(L^{3}\log L)$ and $\mathcal{O}(L^2(\log L)^2)$, respectively, when computing a spherical harmonics transform with $\ell \leq L$. The method \cite{slevinsky2019fast} additionally has a one-time precomputation cost $\mathcal{O}(L^3 \log L)$, which does not affect the overall complexity of Algorithms~\ref{algoBH} and \ref{algoB}. Note that \cite{slevinsky2019fast} has the asymptotic complexity assumed in our complexity analysis in the proof of Theorem~\ref{thm:main}, whereas \cite{bonev2023spherical} has slower asymptotical speed, but is in practice fast for the volume sizes considered here.} We note that there are many other algorithms for fast spherical harmonics transforms, and an alternative includes for example \cite[Chapter~9.6]{plonka2018numerical}. The experiments used accuracy $\varepsilon = 10^{-7}$. The results are shown in Figure~\ref{fig:timings}. We also show the time of each step of Algorithm~\ref{algoBH} in Table~\ref{table:timings}.

We remark that Variant~2 in Figure~\ref{fig:timings} is slower than Variant~1, even though the spherical harmonics transform used in Variant~2 has asymptotically lower complexity than the one used in Variant~1. This has two plausible explanations: the complexity of the spherical harmonics transform in \cite{slevinsky2019fast} is observed to achieve the scaling $\mathcal{O}(L^2 (\log L)^2)$ for large values of $L$, which potentially are not encountered in these numerical tests and is a well-known drawback for fast spherical harmonics transforms. Moreover, the implementation of the spherical harmonics transform in \cite{slevinsky2019fast} is written in C and called in a wrapper package in Julia, and is in turn called from within Python in our implementation of Algorithms~\ref{algoBH} and \ref{algoB}, which could potentially lead to runtimes increased by a constant factor.

\begin{table}
\centering
        \caption{Timings of the steps of Algorithm~\ref{algoBH}. ``Variant 1'' uses \cite{bonev2023spherical} for the fast spherical harmonics transform in step 2 of Algorithm~\ref{algoBH}, and ``Variant 2'' uses \cite{slevinsky2019fast}. The timings of Algorithm~\ref{algoB} are similar.}
    \label{table:timings}
\resizebox{\textwidth}{!}{
\begin{tabular}{r||ccc||ccc}
& & Variant 1 & & & Variant 2 &\\
$N$ & Step 1 & Step 2 & Step 3 & Step 1 & Step 2 & Step 3 \\
\hline
32 &  1.635e-01 &  1.470e-01 &  9.175e-03 &  1.726e-01 &  1.872e+01 &  3.637e-02  \\
48 &  2.539e-01 &  3.326e-01 &  2.387e-02 & 3.478e-01 &  2.948e+01 &  1.029e-01  \\
56 & 4.233e-01 &  5.082e-01 &  3.568e-02 & 4.726e-01 &  4.038e+01 &  1.220e-01  \\
64 &  5.644e-01 &  3.799e+00 &  4.923e-02 & 8.495e-01 &  4.275e+01 &  3.368e-01  \\
128 &  3.340e+00 &  1.058e+01 &  3.183e-01 &   5.168e+00 &  5.349e+02 &  2.300e+00 \\
256 &  2.431e+01 &  4.142e+01 &  2.475e+00 &  3.711e+01 &  2.626e+03 &  3.171e+01 \\
512 &  1.810e+02 &  2.016e+02 &  2.411e+01 &  2.885e+02 &  2.107e+04 &  2.666e+02 \\
\hline
\end{tabular}
}
\end{table}

\begin{figure}
    \centering
\includegraphics[width=\textwidth]{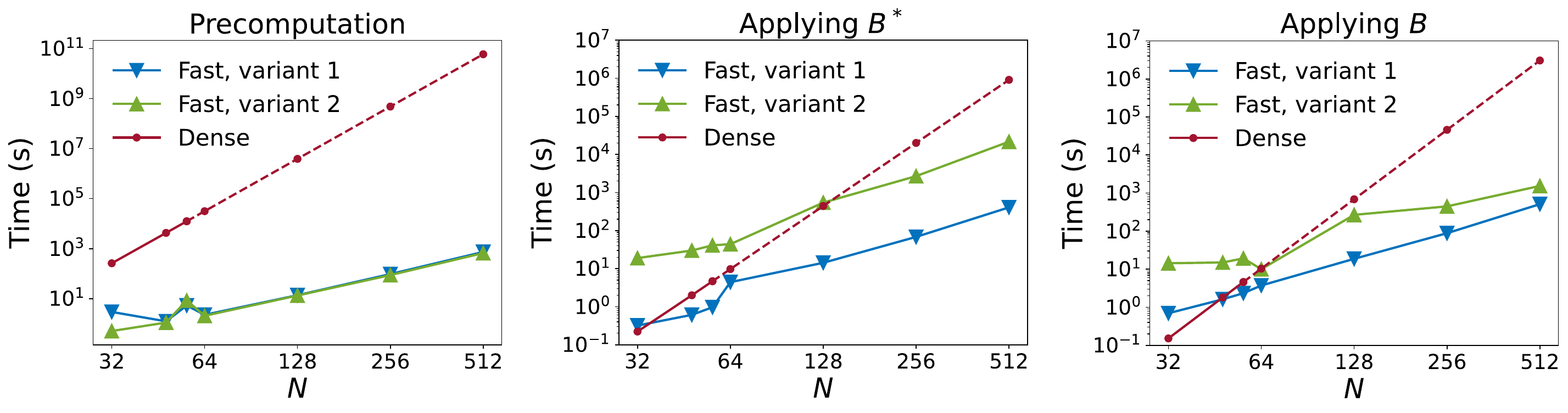}
    \caption{Timings of Algorithm~\ref{algoBH} and Algorithm~\ref{algoB}. ``Variant 1'' uses \cite{bonev2023spherical} for the fast spherical harmonics transform in steps 2 of Algorithm~\ref{algoBH} and Algorithm~\ref{algoB}, and ``Variant 2'' uses \cite{slevinsky2019fast}.}
    \label{fig:timings}
\end{figure}

\section{Conclusion}
This paper developed fast methods to transform between a voxel representation of a three-dimensional function and its expansion into ball harmonics. We proved accuracy guarantees and demonstrated the performance of the methods in numerical experiments. To our knowledge, there are no existing transforms from a voxel representation of a 3D function into an analogous steerable basis that provide
similar approximation with comparable time complexity and theoretical guarantees. 

In future work, we plan to apply these fast ball harmonic transforms in pipelines for computationally demanding 3D data processing tasks, such as reconstruction algorithms in cryo-electron microscopy.  It could also be interesting to adapt the approach here to other geometric domains, such as spherical annuli, which would involve identifying analytic analogues to the plane-wave expansion.

\subsection*{Acknowledgments}
We thank Gregory Chirikjian and Jeremy Hoskins for their helpful comments on a draft of this paper. 
Some of this research was performed while authors JK, NFM, and AS were visiting the long program on Computational Microscopy (Fall 2022) at the Institute for Pure and Applied Mathematics, which is supported by NSF DMS 1925919.
JK was supported in part by NSF DMS 2309782 and CISE-IIS 2312746, and start-up
grants from the College of Natural Sciences and Oden Institute at UT Austin. NFM was supported in part by a start-up grant from Oregon State University. AS and OM were supported in part by AFOSR FA9550-20-1-0266 and FA9550-23-1-0249, the Simons Foundation Math+X Investigator Award, NSF DMS 2009753, and NIH/NIGMS R01GM136780-01.

\bibliographystyle{siamplain}
\bibliography{references_arxiv}

\begin{appendix}
\section{Proof of Theorem \ref{thm:main}}\label{appendix:proofs}

The proof of Theorem \ref{thm:main} is divided into Claims \ref{claim1} and \ref{claim2}, which establish the accuracy guarantee for Algorithm \ref{algoBH} and \ref{algoB}, respectively, followed by Claim \ref{claim3} which establishes the computational complexity for both algorithms. These results rely on technical lemmas established in \S \ref{sec:technical_lemmas}.

\subsection{Accuracy bounds for Algorithm \ref{algoBH}}\label{sec:accuracy_of_BH}

\begin{claim}\label{claim1}
Assume $\lambda \leq 6^{1/3}\pi^{2/3} \lfloor (V^{1/3}+1)/2 \rfloor$ and $\varepsilon >0$ is a user-specified accuracy parameter satisfying $ |\log_2 \varepsilon| \leq 5.3V^{1/3}$. Then Algorithm~\ref{algoBH} produces output  $\widetilde{B}^* f$ satisfying 
$$
    \| \widetilde{B}^*f - B^*f\|_{\ell^\infty} \leq \varepsilon \|f\|_{\ell^1}, 
$$
for all inputs $f$.
\end{claim}

\begin{proof}[Proof of Claim~\ref{claim1}]
In this step, we will use the observation 
\begin{equation}\label{eq:sum_ws}
\begin{split}
\sum_{s=0}^S \sum_{t=0}^{S-1} w_s =4 \pi,
\end{split}
\end{equation}
which follows from the exactness of the quadrature rule \eqref{clenshaw-curtis-equality} and the fact that $Y^0_0(\gamma)$ is a constant function.  We will also use the observation
\begin{equation}\label{eq:cs_ws_Ylm}
\begin{split}
     \sum_{s=0}^S \sum_{t=0}^{S-1} | w_s \overline{Y_{\ell_i}^{m_i}(\gamma_{s,t}}) | &\leq \left( \sum_{s=0}^S \sum_{t=0}^{S-1} 
Y_{\ell_i}^{m_i}(\gamma_{s,t})\overline{Y_{\ell_i}^{m_i}(\gamma_{s,t})} 
w_s\right)^{1/2} \!\!\! \left( \sum_{s=0}^S \sum_{t=0}^{S-1} w_s  \right)^{1/2} \\
&=\sqrt{4 \pi},
\end{split}
\end{equation}
which follows from Cauchy-Schwarz, the fact that $w_s \geq 0$ from \eqref{eq:ws_pos} and exactness of the quadrature rule \eqref{clenshaw-curtis-equality}.

Denote $\alpha = \widetilde{B}^*f$. By writing out the results of the steps of
the algorithm and including the error terms induced by the fast and approximate applications of the NUFFT, spherical harmonics transform, and the interpolation steps, respectively, we have 
\begin{multline*}
\alpha_i =  c_i h^{3/2}  \left( \sum_{q=0}^{Q-1} \left( \frac{ \imath^{\ell_i}}{4\pi}
\left( \sum_{s=0}^S\sum_{t=0}^{S-1} \left( \sum_{j=1}^V f_j e^{-\imath x_j \cdot \rho_{q}\gamma_{s,t}} + \delta^\texttt{nuf}_{q s t} \right) \right. \right.\right. \\
\left. \left. \left. \times w_s \overline{Y_{\ell_i}^{m_i}(\gamma_{s,t})} \right) + \delta_{q}^{\texttt{fsh}} \right)
u_q(\lambda_i) +  \delta^\texttt{in}_i\right),
\end{multline*}
where $\delta^\texttt{nuf}_{q s t}$, $\delta^\texttt{fsh}_{q }$ and $\delta^{\texttt{in}}_i$ denote the error from
the NUFFT, spherical harmonics transform, and fast interpolation, respectively. We will also denote by $\varepsilon^\texttt{nuf}$, $\varepsilon^\texttt{fsh}$ and $\varepsilon^\texttt{in}$ the relative
error parameters for the NUFFT, fast spherical harmonic transform, and fast interpolation, respectively. The error induced by the fast application of the NUFFT satisfies the
$\ell^1$-$\ell^\infty$ relative error bound \cite{barnett2021aliasing, barnett2019parallel}
\begin{equation} \label{err:nuf}
\|\delta^\texttt{nuf}\|_{\ell^\infty}
\le \varepsilon^\texttt{nuf} \sum_{j=1}^V |f_j e^{-\imath x_j \cdot \rho_{q}\gamma_{s,t}}| = \varepsilon^\texttt{nuf} \|f\|_{\ell^1}.
\end{equation}
Next, assume $\varepsilon^\texttt{nuf} \le 1$ (which we will ensure holds below), which together with \eqref{eq:bound_ws} shows that the error of the fast spherical harmonics step satisfies the $\ell^1$ -- $\ell^\infty$ error bound \cite{keiner2008fast,keiner2009using}
\begin{equation}\label{eq:deltafsh_to_epsilonfsh}
\begin{split}
\|\delta^{\texttt{fsh}}\|_{\ell^\infty} &\leq \frac{ \varepsilon^{\texttt{fsh}}}{4\pi}
\sum_{s=0}^S\sum_{t=0}^{S-1} \left| w_s\left(\sum_{j=1}^V f_j e^{-\imath x_j \cdot \rho_{q}\gamma_{s,t}} + \delta^\texttt{nuf}_{q s t} \right)\right| \\
&\leq   \frac{ \varepsilon^{\texttt{fsh}}}{4\pi}
\sum_{s=0}^S\sum_{t=0}^{S-1} w_s (1+\varepsilon^{\texttt{nuf}})\|f\|_{\ell^1}
\leq 2\varepsilon^{\texttt{fsh}} \| f \|_{\ell^1}.
\end{split}
\end{equation}
where the last inequality used \eqref{eq:sum_ws}, \eqref{eq:ws_pos} and that $\varepsilon^{\texttt{nuf}} \leq 1$. Assuming $Q\varepsilon^{\texttt{in}} \leq 1$ (which we will choose to hold below), the error of the fast interpolation step satisfies 
(see \cite{dutt1996} also see \S\ref{claim3} for further discussion of fast interpolation methods) 
\begin{align} 
\|\delta^\texttt{in}\|_{\ell^\infty}
&\le \varepsilon^\texttt{in} \sum_{q=0}^{Q-1} \left| \frac{\imath^{\ell_i}}{4\pi}
\sum_{s=0}^{S} \sum_{t=0}^{S-1} 
 \left( \sum_{j=1}^V f_j e^{-\imath x_j \cdot \rho_{q}\gamma_{s,t}} + \delta^\texttt{nuf}_{q st} \right) w_s \overline{Y_{\ell_i}^{m_i}(\gamma_{s,t})} + \delta_q^\texttt{fsh}\right| \nonumber
\\
&\le \varepsilon^\texttt{in} \sum_{q=0}^{Q-1} \left( \frac{1}{4\pi}
\sum_{s=0}^{S} \sum_{t=0}^{S-1} 
(1+\varepsilon^{\texttt{nuf}})\|f\|_{\ell^1} \left| w_s \overline{Y_{\ell_i}^{m_i}(\gamma_{s,t})}\right| + \delta_q^\texttt{fsh} \right) \label{err:fst}
\\
&\le
\varepsilon^\texttt{in} Q \|f\|_{\ell^1} + 2\varepsilon^{\texttt{in}}\varepsilon^{\texttt{fsh}}Q\|f\|_{\ell^1} \nonumber \le
\varepsilon^\texttt{in} Q \|f\|_{\ell^1} +  2\varepsilon^{\texttt{fsh}}\|f\|_{\ell^1}.\nonumber
\end{align}
 Note that the third inequality in \eqref{err:fst} used \eqref{eq:cs_ws_Ylm}.
 
Next, denote by
$\alpha'_i$ the result of applying Algorithm~\ref{algoBH} without the error induced by the approximations in the NUFFT, spherical harmonics transform and fast
interpolation steps, i.e., 
\begin{equation}\label{eq:def_alpha_in_proof}
\alpha'_i =  c_i h^{3/2} \frac{ \imath^{\ell_i}}{4\pi} \sum_{q=0}^{Q-1} 
\sum_{s=0}^S\sum_{t=0}^{S-1}  \sum_{j=1}^V f_j e^{-\imath x_j \cdot \rho_{q}\gamma_{s,t}}  w_s \overline{Y_{\ell_i}^{m_i}(\gamma_{s,t})} 
u_q(\lambda_i).
\end{equation}

This results in
\begin{equation}
\begin{split}
|\alpha_i - \alpha'_i| &\le  c_i h^{3/2} \left( \sum_{q=0}^{Q-1} \left( \frac{1}{4\pi}  \|\delta^\texttt{nuf}\|_{\ell^\infty} \left( \sum_{s=0}^{S}\sum_{t=0}^{S-1} w_s |Y_\ell^m(\gamma_{s,t})|\right) + \|\delta^{\texttt{fsh}}\|_{\ell^\infty}\right)\right. \\
& \left. \qquad \qquad \qquad \qquad \times u_q(\lambda_i) + \|\delta^\texttt{in}\|_{\ell^\infty}\right) \\
&\le c_i h^{3/2} \left( \left( \frac{1}{2\sqrt{\pi}} \|\delta^\texttt{nuf}\|_{\ell^\infty} + \|\delta^{\texttt{fsh}}\|_{\ell^\infty} \right) \sum_{q=0}^{Q-1} |u_q(\lambda_i)| + \|\delta^\texttt{in}\|_{\ell^\infty} \right) \\
&\le  \pi^2 (3/2)^{1/4} \left( \left( \frac{\|\delta^\texttt{nuf}\|_{\ell^\infty}}{2\sqrt{\pi}}  + \|\delta^{\texttt{fsh}}\|_{\ell^\infty} \right) (2 + \frac{\pi}{2}
\log Q) + \|\delta^\texttt{in}\|_{\ell^\infty} \right), 
\end{split}
\end{equation}
where the last inequality used $\sum_{q=0}^{Q-1} |u_q(\lambda_i)| \leq (2 + \frac{\pi}{2}
\log Q)$ (see \cite[Eq.~(11)]{rokhlin1988fast}) and 
$c_ih^{3/2} \leq \pi^2 (3/2)^{1/4},$
from Lemma \ref{lem:cj}. Combining this inequality with \eqref{err:nuf}, \eqref{eq:deltafsh_to_epsilonfsh} and
\eqref{err:fst} gives
\begin{equation*} \label{nuf+fst}
\|\alpha - \alpha'\|_{\ell^\infty} \le \pi^2 \left(\frac{3}{2}\right)^{1/4} \left( \left( \frac{\varepsilon^\texttt{nuf}}{2\sqrt{\pi}} +  2\varepsilon^{\texttt{fsh}}\right) \left(2 + \frac{\pi}{2} \log Q \right) +  Q\varepsilon^\texttt{in} +  2\varepsilon^\texttt{fsh}  \right) \|f\|_{\ell^1}.
\end{equation*}
Choosing the parameters \begin{equation} \label{eq:nufst}
\varepsilon^\texttt{nuf} = \frac{
\varepsilon/(2\pi^{3/2} (3/2)^{1/4})}{(2 + \frac{\pi}{2} \log Q)} ,\,\,  \varepsilon^\texttt{in} = \frac{\varepsilon/(4\pi^{2}(3/2)^{1/4})}{
Q},\,\,  \varepsilon^\texttt{fsh} = \frac{\varepsilon/(8\pi^{2}(3/2)^{1/4})}{3 +\frac{\pi}{2}\log Q}
\end{equation}
therefore results in
\begin{equation} \label{errp1}
\|\alpha - \alpha'\|_{\ell^\infty} \le \frac{3\varepsilon}{4} \|f\|_{\ell^1}.
\end{equation}
We lastly relate $\alpha'_i$ to $(B^*f)_i$. The definition \eqref{eq:def_alpha_in_proof} gives
\begin{equation}
\begin{split}
\alpha'_i &= c_i h^{3/2} \frac{ \imath^{\ell_i}}{4\pi} \sum_{q=0}^{Q-1} 
\sum_{s=0}^S\sum_{t=0}^{S-1}  \sum_{j=1}^V f_j e^{-\imath x_j \cdot \rho_{q}\gamma_{s,t}}  w_s \overline{Y_{\ell_i}^{m_i}(\gamma_{s,t})} 
u_q(\lambda_i) \\
&= c_i h^{3/2} \sum_{j=1}^V f_j \left(j_{\ell_i}(r_j \lambda_i) \overline{Y_{\ell_i}^{m_i}(\gamma_{x_j})} + \delta_{i,j}^\texttt{dis} \right) = (B^* f)_i + c_i h^{3/2} \sum_{j=1}^V f_j \delta_{i,j}^\texttt{dis},
\end{split}
\end{equation} 
where the second equality follows from Lemma \ref{lem:dis} and $\delta^\texttt{dis}_{i,j}$ is a discretization error that satisfies $\|\delta^\texttt{dis}\|_{\ell^\infty} \le
(3 + \frac{\pi}{2} \log Q)
\varepsilon^\texttt{dis}$. This results in
\begin{equation}  \label{eq:bound_alpha_tilde_to_BH}
|\alpha'_i - (B^*f)_i | \le c_i h^{3/2} \|f\|_{\ell^1} \|\delta^\texttt{dis}\|_{\ell^\infty} \le \pi^2 (3/2)^{1/4} 
(3 + \frac{\pi}{2} \log Q)
\varepsilon^\texttt{dis} \|f\|_{\ell^1}.
\end{equation}
Set
\begin{equation} \label{eq:epsdis}
\varepsilon^\texttt{dis} = \left(\pi^2 (3/2)^{1/4}\left(3 + \frac{\pi}{2} \log{\left\lceil  5.3 V^{1/3}   \right\rceil }\right)\right)^{-1} \frac{\varepsilon}{4}.
\end{equation}
Note that $\varepsilon^\texttt{dis} \leq \varepsilon$ so the assumption $5.3V^{1/3} \geq |\log_2 \varepsilon|$ then implies $$Q = \log \lceil \max\{ 5.3 V^{1/3} , \log_2 \varepsilon^\texttt{dis}\}\rceil = \log\lceil 5.3V^{1/3}\rceil.$$  Combining \eqref{eq:bound_alpha_tilde_to_BH} and \eqref{errp1} then gives
$
\|\alpha - B^* f \|_{\ell^\infty} = \|\widetilde{B}^* f - B^* f \|_{\ell^\infty} \le \varepsilon \|f\|_{\ell^1} .
$
\end{proof}

\subsection{Accuracy bounds for Algorithm \ref{algoB}}\label{sec:accuracy_of_B}

\begin{claim}\label{claim2}
Assume $\lambda \leq 6^{1/3}\pi^{2/3} \lfloor (V^{1/3}+1)/2 \rfloor$ and $\varepsilon >0$ is a user-specified accuracy parameter satisfying $ |\log_2 \varepsilon| \leq 5.3V^{1/3}$. Then Algorithm~\ref{algoB} produces output  
$\widetilde{B} \alpha $ that satisfies the accuracy bounds 
$$
      \| \widetilde{B}\alpha - B\alpha\|_{\ell^\infty} \leq \varepsilon \|\alpha\|_{\ell^1} ,
$$
for all inputs $\alpha$. 
\end{claim}

\begin{proof}[Proof of Claim~\ref{claim2}]

Denote $f = \widetilde{B} \alpha$. By writing out the results of the steps of
the algorithm and including the error terms induced by the fast and approximate applications of the interpolation step, the spherical harmonics transform and the NUFFT, respectively, we have 
\begin{equation*}
\begin{split}
f_j = \sum_{q=0}^{Q-1} & \left( \sum_{s=0}^{S}\sum_{t=0}^{S-1} \left( \sum_{\ell=0}^{L}  \sum_{m = -\ell}^{\ell} \left( \sum_{\substack{i: \ell_i = \ell,\\ \quad m_i = m}} c_i 
\alpha_{i}
 u_q(\lambda_i)  h^{3/2} + \delta_{q,\ell, m}^\texttt{in} \right) \right. \right. \\
 & \qquad \left. \left.  \times \frac{(-\imath)^\ell}{4\pi} w_sY_\ell^m(\gamma_{s,t}) \right) + \delta^{\texttt{fsh}}_{qst} \right)
 e^{\imath x_j \cdot \rho_q \gamma_{s,t}} + \delta_j^\texttt{nuf},
 \end{split}
\end{equation*}
where $\delta_{q,\ell, m}^\texttt{in}$, $\delta_{q,s,t}^\texttt{fsh}$ $\delta_j^\texttt{nuf}$  denote the error from
the fast interpolation, fast spherical harmonics transform and NUFFT, respectively. We will also denote by $\varepsilon^\texttt{in}$, $\varepsilon^\texttt{fsh}$ and $\varepsilon^\texttt{nuf}$ the relative
error parameters for the fast interpolation, fast spherical harmonic transform and NUFFT,   respectively. The error induced by the fast interpolation satisfies the
$\ell^1$-$\ell^\infty$ relative error bound (see \cite{dutt1996} and see also \S\ref{claim3} for further discussion of fast interpolation methods) so
\begin{equation}\label{eq:bound_dfin_for_B}
    \|\delta^\texttt{in}_{\ell, m}\|_{\ell^\infty} \le \varepsilon^\texttt{in} \sum_{\substack{i :
\ell_i = \ell,\\ \quad m_i = m}} c_i h^{3/2} |\alpha_i| \le \pi^2 (3/2)^{1/4} \varepsilon^\texttt{in} \sum_{\substack{i :
\ell_i = \ell,\\ \quad m_i = m}} |\alpha_i|,
\end{equation}
where $\delta_{\ell, m}^\texttt{in}$ is the vector defined by $\delta_{\ell, m}^\texttt{in} = (\delta_{q,\ell, m}^\texttt{in})_{q =0}^{Q-1}$. It follows that
$$
\sum_{\ell=0}^L \sum_{m=-\ell}^\ell \| \delta^\texttt{in}_{\ell, m} \|_{\ell^{\infty}} \leq \pi^2 (3/2)^{1/4} \varepsilon^\texttt{in} \|\alpha \|_{\ell^1}.
$$
Furthermore, the fast spherical harmonics error \cite{keiner2008fast,keiner2009using} satisfies 
\begin{align*}
\|\delta_{s,t}^\texttt{fsh}\|_{\ell^\infty} & \leq \varepsilon^\texttt{fsh}
\left| \sum_{\ell=0}^{L} \sum_{m = -\ell}^{\ell} \left( \sum_{\substack{i: \ell_i = \ell,\\ \quad m_i = m}} c_i 
\alpha_{i}
 u_q(\lambda_i)  h^{3/2} + \delta_{q,\ell, m}^\texttt{in} \right)  \frac{(-\imath)^\ell}{4\pi} w_s  \right| \\
 &\leq \frac{\varepsilon^\texttt{fsh}}{4\pi} \sum_{i=1}^n |c_i \alpha_i u_q(\lambda_i)h^{3/2}w_s| + \frac{\varepsilon^\texttt{fsh}}{4\pi} \sum_{\ell = 0}^L \sum_{m=-\ell}^{\ell}|\delta^{\texttt{in}}_{\ell, m}w_s| \\
 & \leq \frac{\varepsilon^\texttt{fsh}}{4\pi} \pi^2(3/2)^{1/4}w_s\sum_{i=1}^n|u_q(\lambda_i)\alpha_{i}| + \varepsilon^{\texttt{in}}\frac{\varepsilon^\texttt{fsh}}{4\pi}\|\alpha\|_{\ell^1} \pi^2 (3/2)^{1/4} w_s ,
\end{align*}
 where the third inequality used Lemma~\ref{lem:cj}, \cite[Lemma~A.2]{marshall2023fast}, \eqref{eq:bound_ws} and \eqref{eq:bound_dfin_for_B}. Using \eqref{eq:sum_ws}, it follows that
$$
\sum_{s=0}^S \sum_{t=0}^{S-1} \|\delta^\texttt{fsh}_{s,t}\|_{\ell^\infty} \leq \varepsilon^\texttt{fsh} \pi^2(3/2)^{1/4}\sum_{i=1}^n|u_q(\lambda_i)\alpha_{i}| + \varepsilon^{\texttt{in}}\varepsilon^\texttt{fsh}\|\alpha\|_{\ell^1} \pi^2 (3/2)^{1/4}.
$$
 
 Lastly, the error for the NUFFT \cite{barnett2021aliasing, barnett2019parallel} satisfies 
\begin{align*}
\|\delta^\texttt{nuf}\|_{\ell^\infty} &\le \varepsilon^\texttt{nuf}
\left( \sum_{q=0}^{Q-1}\sum_{s=0}^{S}\sum_{t=0}^{S-1}  \left( \sum_{\ell=0}^{L}  \sum_{m = -\ell}^{\ell} \left( \sum_{\substack{i: \ell_i = \ell,\\ \quad m_i = m}} \left| c_i 
\alpha_{i}
 u_q(\lambda_i)  h^{3/2}\right| + \left|\delta_{q, \ell, m}^\texttt{in} \right| \right) \right. \right. \\
 & \qquad\qquad\qquad \left.\left. \times  \frac{1}{4\pi} \left| w_sY_\ell^m(\gamma_{s,t})\right| \right) + \left| \delta_{qst}^{\texttt{fsh}}\right| \right)\\ 
 & \leq \frac{\varepsilon^\texttt{nuf}}{\sqrt{4\pi}}
 \sum_{q=0}^{Q-1} \sum_{\ell=0}^{L} \left( \sum_{m = -\ell}^{\ell} \left( \sum_{\substack{i: \ell_i = \ell,\\ \quad m_i = m}} c_i 
\left|\alpha_{i}
 u_q(\lambda_i) \right| h^{3/2} + \left| \delta_{q, \ell, m}^\texttt{in} \right| \right) \right) \\
 &\qquad \qquad +\varepsilon^{\texttt{nuf}}\sum_{q=0}^{Q-1} \sum_{s=0}^S \sum_{t=0}^{S-1} |\delta^{\texttt{fsh}}_{qst}|\\
& \le  \frac{\varepsilon^\texttt{nuf}}{\sqrt{4\pi}}  \left( \sum_{i=1}^n
\left( \sum_{q=0}^{Q-1} 
|u_q(\lambda_i)| \right) c_i h^{3/2} |\alpha_i |  + \sum_{q=0}^{Q-1}  \pi^2 (3/2)^{1/4} \varepsilon^\texttt{in} \|\alpha \|_{\ell^1} \right) \\
& \qquad + \varepsilon^\texttt{nuf} \sum_{q=0}^{Q-1} \left( \varepsilon^\texttt{fsh} \pi^2(3/2)^{1/4}\sum_{i=1}^n|u_q(\lambda_i)\alpha_{i}| + \varepsilon^{\texttt{in}}\varepsilon^\texttt{fsh}\|\alpha\|_{\ell^1} \pi^2 (3/2)^{1/4} \right) \\
&\le \frac{\varepsilon^\texttt{nuf}}{\sqrt{4\pi}} \left( \left( 2 + \frac{\pi}{2} \log Q \right) \pi^2 (3/2)^{1/4} \|\alpha\|_{\ell^1} +  Q \pi^2 (3/2)^{1/4}  \varepsilon^\texttt{in} \|\alpha\|_{\ell^1} \right) \\
&\qquad + \varepsilon^\texttt{nuf}  \left( \varepsilon^\texttt{fsh} \pi^2(3/2)^{1/4}(2+\frac{\pi}{2}\log Q)\|\alpha\|_{\ell^1} + Q\varepsilon^{\texttt{in}}\varepsilon^\texttt{fsh}\|\alpha\|_{\ell^1} \pi^2 (3/2)^{1/4} \right) \\
&\le  \frac{\varepsilon^\texttt{nuf}}{\sqrt{4\pi}} \pi^2 (3/2)^{1/4} \left( \left( 2 + \frac{\pi}{2} \log Q \right)  +  1\right)\left( 1+ \varepsilon^{\texttt{fsh}}\sqrt{4\pi}\right) \|\alpha\|_{\ell^1} ,
\end{align*}
where the second equality uses \eqref{eq:cs_ws_Ylm}, the fourth uses \cite[Eq.~(11)]{rokhlin1988fast} and the fifth assumes $Q  \varepsilon^\texttt{in} \le 1$, which we will show holds below. 

Next, denote by
$\tilde{f}$ the result of applying Algorithm~\ref{algoB} without the error induced by the approximations in the fast
interpolation, spherical harmonics transform and NUFFT steps, i.e.,
\begin{equation}\label{eq:def_f_in_proof}
f'_j = \sum_{q=0}^{Q-1}\sum_{s=0}^{S}\sum_{t=0}^{S-1} \sum_{\ell=0}^{L}  \sum_{m = -\ell}^{\ell} \sum_{\substack{i:  \ell_i = \ell, \\ \quad m_i = m}}  c_i 
\alpha_{i}
 u_q(\lambda_i)  h^{3/2} \frac{(-\imath)^\ell}{4\pi} w_sY_\ell^m(\gamma_{s,t}) 
 e^{\imath x_j \cdot \rho_q \gamma_{s,t}}  .
\end{equation}
We have
\begin{align*}
|f_j - f'_j| &\le \left( \sum_{q=0}^{Q-1}  \sum_{s=0}^{S}\sum_{t=0}^S\left(
 \sum_{\ell=0}^{L} \sum_{m=-\ell}^{\ell}\|\delta_{\ell, m}^\texttt{in}\| _{\ell^\infty}\frac{1}{4\pi}|w_s Y_\ell^m(\gamma_{s,t})| \right) + |\delta^{\texttt{fsh}}_{qst}| \right) +
|\delta_j^\texttt{nuf}| \\ 
&\le \frac{1}{\sqrt{4\pi}}\left( \sum_{q=0}^{Q-1} \pi^2 (3/2)^{1/4} \varepsilon^\texttt{in} \|\alpha\|_{\ell^1} \right) \\
&\qquad + \sum_{q=0}^{Q-1}\left( \varepsilon^\texttt{fsh} \pi^2(3/2)^{1/4}\sum_{i=1}^n|u_q(\lambda_i)\alpha_{i}| + \varepsilon^{\texttt{in}}\varepsilon^\texttt{fsh}\|\alpha\|_{\ell^1} \pi^2 (3/2)^{1/4} \right) \\
&\qquad + \frac{\varepsilon^\texttt{nuf}}{\sqrt{4\pi}} \pi^2 (3/2)^{1/4} \left( \left( 2 + \frac{\pi}{2} \log Q \right)  +  1\right)\left( 1+ \varepsilon^{\texttt{fsh}}\sqrt{4\pi}\right) \|\alpha\|_{\ell^1}
 \\
 &\le \frac{1}{\sqrt{4\pi}}Q \pi^2 (3/2)^{1/4} \varepsilon^\texttt{in} \|\alpha\|_{\ell^1} \\
&\qquad +  \varepsilon^\texttt{fsh} \pi^2(3/2)^{1/4}(2+\frac{\pi}{2}\log Q) \|\alpha \|_{\ell^1} + Q\varepsilon^{\texttt{in}}\varepsilon^\texttt{fsh}\|\alpha\|_{\ell^1} \pi^2 (3/2)^{1/4}   \\
&\qquad + \frac{\varepsilon^\texttt{nuf}}{\sqrt{4\pi}} \pi^2 (3/2)^{1/4} \left( \left( 2 + \frac{\pi}{2} \log Q \right)  +  1\right)\left( 1+ \varepsilon^{\texttt{fsh}}\sqrt{4\pi}\right) \|\alpha\|_{\ell^1}
 \\
&\le \pi^2 (3/2)^{1/4}
\left(  \frac{\varepsilon^\texttt{in} Q}{\sqrt{4\pi}} + \varepsilon^{\texttt{fsh}}(3+\frac{\pi}{2}\log Q) +
\frac{2\varepsilon^\texttt{nuf}}{\sqrt{4\pi}}  \left( 3 + \frac{\pi}{2} \log Q \right)     \right) \|\alpha\|_{\ell^1},
\end{align*}
where the last equation used $Q\varepsilon^\texttt{in} \leq 1$ and $\sqrt{4\pi}\varepsilon^\texttt{in} \leq 1$, which we will show holds below.
Choosing the parameters
\begin{equation}\label{eq:error_nufft_B_proof}
\varepsilon^\texttt{in} = \frac{1}{4Q}\frac{2\sqrt{\pi}\varepsilon}{ \pi^2 (3/2)^{1/4} } ,  \,\,
\varepsilon^\texttt{fsh} = \frac{1}{4}\frac{\varepsilon/(\pi^2 (3/2)^{1/4}) }{  3 + \frac{\pi}{2} \log Q }, \,\,
 \varepsilon^\texttt{nuf} = \frac{1}{4}\frac{\sqrt{\pi}\varepsilon/(\pi^2 (3/2)^{1/4})}{   3 + \frac{\pi}{2} \log Q } 
\end{equation}
results in
\begin{equation}\label{errp2}
    \|f - f'\|_{\ell^\infty} \le \frac{3\varepsilon}{4} \|\alpha\|_{\ell^1}.
\end{equation}
We lastly relate $f'$ to $B\alpha$. The definition \eqref{eq:def_f_in_proof} gives
\begin{equation}
\begin{split}
f'_j
&= \sum_{q=0}^{Q-1} \sum_{s=0}^{S}\sum_{t=0}^S \sum_{i=1}^n c_i 
\alpha_{i}
 u_q(\lambda_i)  h^{3/2} \frac{(-\imath)^\ell}{4\pi} w_sY_\ell^m(\gamma_{s,t}) 
 e^{\imath x_j \cdot \rho_q \gamma_{s,t}}
\\
&= \sum_{i=1}^n c_i h^{3/2}
\alpha_{i} \sum_{q=0}^{Q-1} \sum_{s=0}^{S}\sum_{t=0}^S  
 u_q(\lambda_i)   \frac{(-\imath)^\ell}{4\pi} w_sY_\ell^m(\gamma_{s,t}) 
 e^{\imath x_j \cdot \rho_q \gamma_{s,t}} \\
&= \sum_{i=1}^m c_i h^{3/2} \alpha_i \left( j_{\ell_i} (r_{x_j} \lambda_i) Y_{\ell_i}^{m_i}(\gamma_{x_j}) + \delta^\texttt{dis}_{i,j} \right) = (B f)_j + \sum_{i=1}^m c_i h^{3/2} \alpha_i \delta^\texttt{dis}_{i,j},
\end{split}
\end{equation}
where the third equality follows from Lemma \ref{lem:dis} and $\delta^\texttt{dis}_{ij}$ is a discretization error that satisfies
$\|\delta^\texttt{dis}\|_{\ell^\infty} \le (3 + \frac{\pi}{2} \log Q)\varepsilon^\texttt{dis}$. This results in
\begin{equation}\label{eq:relate_tildef_andB}
\|f' - B \alpha\|_{\ell^\infty} \le \pi^2 (3/2)^{1/4} (3 + \frac{\pi}{2} \log Q) \varepsilon^\texttt{dis} \|\alpha
\|_{\ell^\infty} .
\end{equation}
Set
\begin{equation} \label{eq:epsdis2}
\varepsilon^\texttt{dis} =  \left(\pi^2 (3/2)^{1/4} \left(3 + \frac{\pi}{2} \log \left\lceil  5.3 V^{1/3}  \right\rceil \right)\right)^{-1} \frac{\varepsilon}{4}.
\end{equation}
Note that $\varepsilon^\texttt{dis} \leq \varepsilon$ so the assumption $5.3V^{1/3} \geq |\log_2 \varepsilon|$ then implies $$Q = \log \lceil \max\{ 5.3 V^{1/3} , \log_2 \varepsilon^\texttt{dis}\}\rceil = \lceil 5.3V^{1/3}\rceil .$$ Combining \eqref{eq:relate_tildef_andB} and \eqref{errp2} gives
$
\|f - B \alpha \|_{\ell^\infty} = \|\widetilde{B} \alpha - B \alpha\|_{\ell^{\infty}}\le \varepsilon \|\alpha\|_{\ell^1} .
$
\end{proof}

\subsection{Bound on computational complexity of Algorithms \ref{algoBH} and \ref{algoB}}\label{sec:complexity_of_BH}

\begin{claim} \label{claim3}
Assume $\lambda \leq 6^{1/3}\pi^{2/3} \lfloor (V^{1/3}+1)/2 \rfloor$ and $\varepsilon >0$ is a user-specified accuracy parameter satisfying $ |\log_2 \varepsilon| \leq 5.3V^{1/3}$. Then Algorithms~\ref{algoBH}--\ref{algoB} involve $\mathcal{O}(V(\log V)^2 + V |\log \varepsilon |^2) $ operations.
\end{claim}

\begin{proof}[Proof of Claim~\ref{claim3}]

The NUFFT with $n$ source points and $m$ target points in $\mathbb{R}^d$ and a precision parameter $0 \leq \varepsilon \leq 1$ takes $\mathcal{O} ( n\log n+m\left( \log \frac{1}{\varepsilon} \right)^d)$
operations to achieve $\ell^1$-$\ell^\infty$ relative error $\varepsilon$ (see \cite[Eq.~(9)]{barnett2021aliasing} and \cite[Sec.~1.1]{barnett2019parallel}). For both Algorithms \ref{algoBH} and \ref{algoB},
the number of source points is $V$ and the number of target points is $S(S+1) Q = \mathcal{O}(V)$. In \eqref{eq:nufst} and \eqref{eq:error_nufft_B_proof}, both algorithms set the precision to $\varepsilon^\texttt{nuf} = \mathcal{O}(\varepsilon/\log Q)$. The computational complexity is therefore $\mathcal{O} (V\log V + V|\log \varepsilon^\texttt{nuf} |^2 ) = \mathcal{O} (V\log V + V|\log \varepsilon - \log \log Q |^2 )=\mathcal{O}(V\log V + V|\log \varepsilon |^2)$
operations.

Algorithms \ref{algoBH} and \ref{algoB} use the fast spherical harmonics transforms defined in \eqref{eq:def_sph_transform1} and \eqref{eq:def_sph_transform2}. Each algorithm performs $Q = \mathcal{O}(V^{1/3})$ applications of these transforms, where each application has $\mathcal{O}(S^2)$ grid points and $\mathcal{O}(L^2)$ basis coefficients. With $\varepsilon^\texttt{fsh}$ from \eqref{eq:nufst} and \eqref{eq:error_nufft_B_proof}, the computations therefore use $\mathcal{O} (V(\log V)^2 + V |\log \varepsilon^{\texttt{fsh}}|^2) = \mathcal{O} (V(\log V)^2 + V |\log \varepsilon|^2)$ operations (see \cite{potts2003fast,keiner2008fast,keiner2009using}).

 For fixed $
 \ell , m$, the algorithms perform polynomial interpolation from the $Q = \mathcal{O}(V^{1/3})$ source points to the target points $\lambda_{\ell k}$, where $1 \leq k \leq K$. There are many fast methods for performing interpolation \cite{Dutt1993,Greengard2004,lee2005type,dutt1996,gimbutas2020fast,plonka2018numerical}. Using the lowest complexity method of these \cite{dutt1996}, the cost of each interpolation to $\ell^1$-$\ell^\infty$ relative precision $\varepsilon^\texttt{in}$ is $\mathcal{O}( V^{1/3} \log V + K |\log \varepsilon^\texttt{in}|)$. By \eqref{eq:bandlimit_on_L} we have $L,K = \mathcal{O}(V^{1/3})$. There are therefore fewer than $L(2L+1) = \mathcal{O}(V^{2/3})$ interpolation problems, which gives a total complexity of $\mathcal{O}(V \log V + V |\log \varepsilon^\texttt{in}|)$.
By \eqref{eq:nufst} and \eqref{eq:error_nufft_B_proof},  the interpolation step therefore uses $\mathcal{O}(V \log V + V |\log \varepsilon|)$ operations.

All three steps of the algorithm have computational complexity $\mathcal{O}(V(\log V)^2 + V |\log \varepsilon|^2)$, which concludes the proof of Theorem~\ref{thm:main}.
\end{proof}

\subsection{Auxiliary Lemmas}\label{sec:technical_lemmas}

\begin{lemma} \label{lem:dis}
Define $S$ and $Q$ as in Lemma \ref{lem:num_angular_nodes} and  \ref{lem:num_radial_nodes} with accuracy parameter $\eta > 0$. Then for any $i \in \{ 1, \ldots , n\}$, and $j\in \{1, \ldots , V\}$
$$
\left| \sum_{q=0}^{Q-1} 
\sum_{s=0}^S\sum_{t=0}^{S-1}  e^{-\imath x_j \cdot \rho_{q}\gamma_{s,t}}  \frac{\imath^\ell w_s}{4\pi} \overline{Y_{\ell_i}^{m_i}(\gamma_{s,t})} 
u_q(\lambda_i) - j_\ell(r_{x_j} \lambda_i)\overline{Y^m_\ell(\gamma_{x_j})} \right| \le   \left(3 + \frac{\pi}{2}\log Q\right)\eta,
$$
\end{lemma}
\begin{proof}
Putting together Lemma~\ref{lem:num_angular_nodes} and \ref{lem:num_radial_nodes}, we have
\begin{align*}
 \sum_{q=0}^{Q-1} 
\sum_{s=0}^S\sum_{t=0}^{S-1}&  e^{-\imath x_j \cdot \rho_{q}\gamma_{s,t}}  w_s \overline{Y_{\ell_i}^{m_i}(\gamma_{s,t})} 
u_q(\lambda_i) = \sum_{q=0}^{Q-1} \left( j_\ell(r_{x_j} \rho_q)\overline{Y^{m_i}_{\ell_i}(\gamma_{x_j})} + \delta_{q i j}^\texttt{ang}  \right) u_q(\lambda_i)  \\
&= j_\ell(r_{x_j} \lambda_i)\overline{Y^{m_i}_{\ell_i}(\gamma_{x_j})} + \delta_{ij}^\texttt{rad} + \sum_{q=0}^{Q-1} \delta_{q i j}^\texttt{ang} u_q(\lambda_i),
\end{align*}
where $\delta_{q i j}^\texttt{ang}$ and $\delta^\texttt{rad}_{ij}$ are the approximation errors from Lemma~\ref{lem:num_angular_nodes} and \ref{lem:num_radial_nodes}, respectively. By the choice of $Q$ and $S$, these approximation errors are controlled by $| \delta_{q i j}^\texttt{ang} | \leq \eta$ and $| \delta^\texttt{rad}_{ij} | \leq \eta$. Bounding the term $\sum_{q=0}^Q |u_q(\lambda_i)|$ as in \cite[Eq.~(11)]{rokhlin1988fast}, 
$$
|\delta_{ij}^\texttt{rad} + \sum_{q=0}^{Q-1} \delta_{q i j}^\texttt{ang} u_q(\lambda_i)| \le \eta + \eta(2 + \frac{\pi}{2} \log Q ),
$$
which finishes the proof.
\end{proof}

\begin{lemma}\label{lem:cj}
If $\lambda_{\ell k} \le 6^{1/3}\pi^{2/3}h^{-1}$, then the normalization constants $c_{\ell k}$ in \eqref{eq:sph_bessel1} satisfy $c_{\ell k} \le \pi^2 (3/2)^{1/4}h^{-3/2}$.
\end{lemma}
\begin{proof}[Proof of Lemma \ref{lem:cj}]
By the same argument as in \cite[Lemma A.4]{marshall2023fast}, we obtain
$$
J_{\ell+1/2}'(\lambda_{\ell k})^2 \ge \left(\frac{2}{\pi \lambda_{\ell k}} \right)^2 \frac{2 \sqrt{\lambda_{\ell k}^2 - (\ell+1/2)^2}}{\pi}.
$$
Since $\lambda_{\ell k} > \ell + 1/2 + k \pi - \pi/2 + 1/2 > \ell + 2.5$ for $(\ell,k) \in \mathbb{Z}_{\ge 0} \times \mathbb{Z}_{>0}$ (see \cite[Eq. 1.6]{elbert2001some}), we have
$\sqrt{\lambda_{\ell k}^2 - (\ell + 1/2)^2} > \sqrt{ 6 }$. Using this and the assumption $\lambda_{\ell k} \le 6^{1/3}\pi^{2/3}h^{-1}$ yields
$$
c_{\ell k} = \frac{2\sqrt{\lambda_{\ell k}}}{\pi^{1/2} |J_{\ell + 1/2}'(\lambda_{\ell k})|} \le \frac{\pi \lambda_{\ell k}^{3/2}}{2^{1/2} 6^{1/4}} \le \pi^2 (3/2)^{1/4}h^{-3/2},
$$ 
as announced.
\end{proof}

\section{Conventions}  \label{notation} 
 This appendix specifies the notational conventions used in the article and includes background material to make the main paper self-contained.  Since conventions for special functions and operators used in this paper vary across mathematical communities, care is taken to avoid any potential ambiguity.

\subsection{Spherical coordinates}  \label{sec:sphericalcoor}
Suppose that $f : \mathbb{R}^3 \rightarrow \mathbb{C}$ and $x = (x_1,x_2,x_3)$ are Cartesian coordinates. Let  $(r,\theta,\phi)$ be spherical coordinates  defined by
\begin{align}\label{eq:def_spherical_coords}
\left\{
\begin{array}{l}
x_1 = r \sin \theta \cos \phi, \\ 
x_2 = r \sin \theta \sin \phi, \\
x_3 = r \cos \theta,
\end{array} \right.
\end{align}
where $r \in [0,\infty)$, $\theta \in [0,\pi]$, and $\phi \in [0,2\pi)$. 
More concisely, we can define $\gamma \in \mathbb{S}^2 = \{ x \in \mathbb{R}^3 : |x| = 1\}$ by
\begin{equation}\label{eq:def_spherical_gamma}
\gamma = (\sin \theta \cos \phi, \sin \theta \sin \phi, \cos \theta),
\end{equation}
and write $x = r \gamma$. When necessary to avoid ambiguity, we write $r_x,\theta_x,\phi_x,\gamma_x$ to denote the spherical coordinates defined in \eqref{eq:def_spherical_coords} and 
\eqref{eq:def_spherical_gamma} corresponding to $x \in \mathbb{R}^3$.
The integral of $f$ is expressed in spherical coordinates as
$$
\int_{\mathbb{R}^3} f(x)\dd x  = 
\int_0^{2\pi} \int_0^\pi \int_0^\infty f(r, \theta,\phi) r^2 \sin \theta \dd r \dd \theta \dd\phi
= \int_{\mathbb{S}^2} \int_0^\infty f(r \gamma) r^2\dd r\dd\sigma(\gamma),
$$
where $\dd\sigma(\gamma)$ is the surface element on $\mathbb{S}^2$. 
The Laplacian $\Delta = \partial_{x_1 x_1} + \partial_{x_2 x_2} + \partial_{x_3 x_3}$ is expressed in spherical coordinates by
\begin{equation} \label{lap}
\begin{split}
\Delta &=  \partial_{rr} + \frac{2}{r} \partial_r + \frac{1}{r^2} \left( \frac{1}{\sin\theta}\partial_\theta \Big(\sin \theta \partial_\theta\Big) + \frac{1}{(\sin \theta)^2}\partial_{\phi \phi} \right) \\
&= \partial_{rr} + \frac{2}{r} \partial_r + \frac{1}{r^2} \Delta_{\mathbb{S}^2},
\end{split}
\end{equation}
where $\Delta_{\mathbb{S}^2}$ is the Laplace-Beltrami operator on the sphere $\mathbb{S}^2= \{ x \in \mathbb{R}^3 : \| x \| = 1\}$.
We define the Fourier-transform of an integrable function $f:\mathbb{R}^3 \rightarrow \mathbb{C}$ using the convention
\begin{equation}
    \widehat{f}(\omega) = \int_{\mathbb{R}^3} f(x)e^{-\imath x\cdot \omega}\dd x .
\end{equation}
With this convention, the inverse Fourier transform is given by
$$
f(x) = \frac{1}{(2\pi)^{3}} \int_{\mathbb{R}^3} \widehat{f}(\omega) e^{i x \cdot \omega}\dd\omega,
$$
when $f$ is sufficiently nice, for example when $f$ is in Schwartz space $\mathcal{S}(\mathbb{R}^3)$.

\subsection{Spherical Bessel functions} \label{sec:spherebessel} Let $j_\ell$ denote the $\ell$-th order spherical Bessel function of the first kind. We have
\begin{equation} \label{jnJN}
j_\ell(r) = \left(\frac{\pi}{2 r} \right)^{1/2} J_{\ell+1/2}(r),
\end{equation}
where $J_q$ is the $q$-th order Bessel function of the first kind, see \cite[Eq. 10.47.3]{dlmf}. Further, $j_\ell$ satisfies the differential equation
\begin{equation} \label{jndiffeq}
- \left( \partial_{rr} + \frac{2}{r}\partial_{r} - \frac{\ell (\ell+1)}{r^2} \right) j_\ell(r) = j_\ell(r),
\end{equation}
see \cite[10.47.1]{dlmf}.

\subsection{Spherical harmonics} \label{sec:sphereharmonic} Let $Y_{\ell}^m(\theta,\phi)$ be (normalized) spherical harmonics defined by
\begin{equation}\label{eq:def_Ylm}
    Y_{\ell}^m(\theta,\phi) = \left( \frac{(\ell - m)! (2 \ell + 1)}{4 \pi ( \ell + m)!} \right)^{\! 1/2} e^{i m \phi} P_\ell^m(\cos \theta),
\end{equation}
for  $m \in \{-\ell,\ldots,\ell\}$ and $\ell \in \mathbb{Z}_{\ge 0}$, see \cite[Eq. 14.30.1]{dlmf}. 
Here, $P_\ell^m$ denotes the 
associated Legendre function of the first kind, see \cite[Eq. 14.3.6]{dlmf}. With the above normalization, the spherical harmonics satisfy the identity
\begin{equation}\label{eq:orthogonality_sph_harm}
    \int_0^{2\pi} \int_0^{\pi} Y_{\ell_1}^{m_1}(\theta,\phi) \overline{Y_{\ell_2}^{m_2}(\theta,\phi)} \sin \theta \dd \theta \dd \phi = \delta_{\ell_1,\ell_2} \delta_{m_1,m_2},
\end{equation}
where $\delta_{\ell_1,\ell_2} = 1$ if $\ell_1 = \ell_2$ and $0$ if $\ell_1 \not = \ell_2$, see \cite[Eq. 14.30.8]{dlmf}. Additionally, the $Y_\ell^m$ satisfy
$$
\sum_{m=-\ell}^{\ell} |Y_\ell^m(\theta, \phi)|^2 = \frac{2\ell+1}{4\pi},
$$
which is known as Uns\"{o}ld's identity \cite[Eq.~(4.37)]{chirikjian2016harmonic}  and follows from the addition theorem 
$$
P_\ell(\gamma_1 \cdot \gamma_2) = \frac{4\pi}{2\ell+1} \sum_{m=-\ell}^{\ell} Y_\ell^m(\gamma_1)\overline{Y_\ell^m(\gamma_2)},
$$
for $\gamma_1, \gamma_2 \in \mathbb{S}^2$.
 See also  \cite[Corollary~IV.2.9b]{stein1971introduction}. In particular, it follows that
for all $\theta$ and $\phi$ \begin{equation}\label{eq:unsold_bound}
    |Y_\ell^m(\theta, \phi)| \leq \sqrt{\frac{2\ell + 1}{4\pi}}, \quad \text{ for any }  \ell, m  \, \text{ with }   -\ell \leq m \leq \ell.
\end{equation}
Further, spherical harmonics satisfy the differential equation
\begin{equation} \label{eqy}
-\left( \frac{1}{\sin\theta}\partial_\theta \Big(\sin \theta \partial_\theta\Big) + \frac{1}{(\sin \theta)^2}\partial_{\phi \phi} \right) Y_{\ell}^m(\theta,\phi) = \ell(\ell+1) Y_{\ell}^m(\theta,\phi),
\end{equation}
see \cite[Eq. 14.30.10]{dlmf}, 
which can be written concisely as
\begin{equation}
   -\Delta_{\mathbb{S}^2} Y_\ell^m = \ell(\ell+1) Y_\ell^m. 
\end{equation}
In other words, the functions $Y_{\ell}^m$ are eigenfunctions of the operator $-\Delta_{\mathbb{S}^2}$ with eigenvalue $\ell(\ell+1)$ for 
 $\ell,m \in \mathbb{Z}$ such that $|m| \le \ell$.

{
 \begin{remark}[Real-valued spherical harmonics]
 Real-valued spherical harmonics $\tilde{Y}_\ell^m: \mathbb{S}^2 \rightarrow \mathbb{R}$ can be obtained by taking linear combinations of $Y_{\ell}^m$ as follows:
 \begin{equation}\label{eq:def_real_Ylm}
 \tilde{Y}_\ell^m   =
 \begin{cases}
     \frac{i}{\sqrt{2}}\left(Y^{-|m|}_\ell - (-1)^m Y^{|m|}_\ell \right) &\quad \text{ if } m < 0 \\
     Y^0_\ell & \quad \text{ if } m = 0 \\
     \frac{1}{\sqrt{2}}\left(Y^{-|m|}_\ell + (-1)^m Y^{|m|}_\ell \right) & \quad \text{ if } m > 0.
 \end{cases}
 \end{equation}
 Indeed, the fact that $\tilde{Y}_\ell^m$ is real-valued follows from the definition \eqref{eq:def_Ylm} of $Y_\ell^m$   and the
 identity $$
 P_\ell^{-m}(x) = (-1)^m \frac{(\ell - m)!}{(\ell + m)!} P_\ell^{m}(x);$$ see \cite[Eq.~(14.9.3)]{dlmf}.  The functions $\tilde{Y}_\ell^m$ are orthonormal on $\mathbb{S}^2$.  
\end{remark}
}
\subsection{Ball harmonics} \label{sec:lapeigedirdef}
The harmonics on the unit ball $\mathbb{B} = \{x \in \mathbb{R}^3 : \|x\| < 1\}$ in $\mathbb{R}^3$ are eigenfunctions of the Dirichlet Laplacian, i.e., they solve the eigenproblem
\begin{equation}
\left\{
\begin{array}{cl}
-\Delta\psi = \lambda \psi &  \text{ in } \mathbb{B}, \\[5pt]
\psi = 0   & \text{ on } \partial \mathbb{B},
\end{array} \right.
\end{equation}
where $\partial \mathbb{B} = \mathbb{S}^2$ denotes the boundary of $\mathbb{B}$. The eigenfunctions of the Laplacian form an orthonormal basis of square-integrable functions on $\mathbb{B}$ and can be expressed by
\begin{equation} \label{eigenfundef}
\psi_{k, \ell, m}(r, \theta , \phi) = c_{\ell k}j_\ell(\lambda_{\ell k} r)  Y_{\ell}^m(\theta,\phi)\chi_{[0,1)}(r),
\end{equation}
for $ m \in \{-\ell,\ldots,\ell\}$, $ \ell \in \mathbb{Z}_{\ge 0}$, and $ k \in \mathbb{Z}_{>0}$,
where $c_{\ell k}$ are positive normalization constants ensuring $\|\psi_{k, \ell, m}\|_{L^2} =1 $, 
and the constant $\lambda_{\ell k}$ is the $k$-th positive root of $j_\ell$; see for example \cite[\S 3.3]{grebenkov2013}. The inclusion of the indicator function $\chi_{[0,1)}(x)$ for $[0,1)$ extends the domain of the ball harmonics to $\mathbb{R}^3$ by setting the value equal to zero outside of $\mathbb{B}$.

Note that by \eqref{jnJN} the constant $\lambda_{\ell k}$ can equivalently be defined as the $k$-th positive root of $J_{\ell+1/2}$ by \eqref{jnJN}. Moreover, it holds
$$
-\Delta \psi_{k, \ell, m} = \lambda_{\ell k}^2 \psi_{k, \ell, m},
$$
so that $\lambda_{\ell k}^2$ is the eigenvalue associated with the eigenfunctions $\psi_{k, \ell, m}$ of the operator $-\Delta$. Indeed, this is straightforward to see by writing the Laplacian in spherical coordinates: 
$$
\Delta = \partial_{rr} + \frac{2}{r} \partial_r + \frac{1}{r^2} \Delta_{\mathbb{S}^2},
$$
where $\Delta_{\mathbb{S}^2}$ is the Laplace Beltrami operator on the sphere. This gives
\begin{equation}
\begin{split}
    \Delta \psi_{ k, \ell, m} &= \lambda_{\ell k}^2 j_\ell''(\lambda_{ \ell k} r) Y_{\ell}^m + \frac{2 \lambda_{\ell k}}{r} j_\ell'(\lambda_{ \ell k} r) Y_{\ell}^m + 
j_{\ell}(\lambda_{\ell k}r) \frac{1}{r^2} \Delta_{\mathbb{S}^2}Y_{\ell}^m \\
& = \left(\lambda_{\ell k}^2 j_\ell''(\lambda_{\ell k} r)  + \frac{2 \lambda_{ \ell k}}{r} j_\ell'(\lambda_{\ell k} r) - \frac{\ell(\ell+1)}{r^2}j_\ell(\lambda_{\ell k}r) \right) Y_{\ell}^m(\theta, \phi) \\
&  = - \lambda_{\ell k}^2 j_\ell(\lambda_{\ell k}r) Y_{\ell}^m = - \lambda_{ \ell k}^2 \psi_{ k, \ell, m},
\end{split}
\end{equation}
where the final equality follows from \eqref{jndiffeq}, and $f' := \partial_r f$ and $f'' := \partial_{rr} f$.
Finally, we define the normalization constants $c_{\ell k}$ for completeness. We have
\begin{equation}
\begin{split}
\|\psi_{k, \ell, m}\|_{L^2}^2 
&=  c_{\ell k}^2\int_0^{2\pi} \int_0^\pi \int_0^1 j_\ell( \lambda_{\ell k}r)^2 |Y_{\ell}^m(\theta,\phi)|^2 r^2 \sin \theta \dd r \dd \theta \dd \phi \\ 
&=  c_{\ell k}^2\int_0^1 j_\ell(\lambda_{\ell k}r)^2  r^2 \dd r 
=  c_{\ell k}^2 \int_0^1 \frac{\pi}{2\lambda_{\ell k}}  J_{\ell+1/2}(\lambda_{\ell k}r)^2 r \dd r \\ 
&= \frac{\pi c_{\ell k}^2}{4\lambda_{\ell k}} (J_{\ell+1/2}'(\lambda_{\ell k}))^2,
\end{split}
\end{equation}
where the final equality follows from \cite[Eq. 10.22.37]{dlmf}.  Thus,
\begin{equation}\label{eq:def_clk}
c_{\ell k} = \frac{2 \sqrt{\lambda_{\ell k}}}{\pi^{1/2} |J_{\ell+1/2}'(\lambda_{\ell k})|}.
\end{equation}

\begin{remark}[Real-valued ball harmonics]
Note that while the ball harmonics $\psi_{k,\ell,m}$ are complex-valued, orthonormal real-valued analogs $\tilde{\psi}_{k,\ell,m}$ can be obtained by defining
 \begin{equation}\label{eq:def_real_psiklm}
 \tilde{\psi}_{k, \ell, m}   =
 \begin{cases}
     \frac{i}{\sqrt{2}}\left(\psi_{k,\ell,-|m|} - (-1)^m \psi_{k,\ell,|m|} \right) &\quad \text{ if } m < 0 \\
     \psi_{k,\ell,0} & \quad \text{ if } m = 0 \\
     \frac{1}{\sqrt{2}}\left(\psi_{k,\ell,-|m|} + (-1)^m \psi_{k,\ell,|m|} \right) & \quad \text{ if } m > 0;
 \end{cases}
 \end{equation}
 the fact that $\tilde{\psi}_{k,\ell,m}$ is real-valued follows from \eqref{eq:def_real_Ylm}.
 \end{remark}

\subsection{Quadrature on $\mathbb{S}^2$} \label{Clenshaw-Curtis-sec}
We define a quadrature rule on the sphere that is exact for integrals corresponding to inner products of functions that have bandlimited spherical harmonic expansions.
We define quadrature nodes $\gamma_{s,t} \in \mathbb{S}^2$ by
\begin{equation} \label{eq:prodquad}
\gamma_{s,t} = (\sin \theta_s \cos \phi_t, \sin \theta_s \sin \phi_t, \cos \theta_s),
\end{equation}
where
\begin{equation}\label{eq:angular_grid}
(\theta_s,\phi_t ) := \left( \frac{\pi s}{S}, \frac{2\pi t}{S} \right) \quad \text{for} \quad (s,t) \in \{ 0, \ldots, S \} \times \{ 0, \ldots ,  S-1 \}.
\end{equation}
We define quadrature weights $w_s$ equal to the Clenshaw-Curtis quadrature with respect to $\theta_s$ and uniform quadrature with respect to $\phi_t$. The quadrature weights $w_s$ are given by
\begin{equation}\label{eq:clenshaw_curtis_weights}
    w_s := \frac{4\pi\varepsilon_s}{S^2} \sum_{u=0}^{\lfloor S/2 \rfloor} \frac{2 \varepsilon_u}{1-4u^2 }\cos \left(\frac{2\pi su }{S}\right),
    \quad \text{for} \quad s \in \{0,\ldots,S\},
\end{equation}
 where $\varepsilon_0 = \varepsilon_S = 1/2 $ and  $\varepsilon_u = 1$ otherwise, see \cite{plonka2018numerical, waldvogel2006fast}. We say that $f : \mathbb{S}^2 \rightarrow \mathbb{C}$ is bandlimited on the sphere if
\begin{equation} \label{bandlimited-on-sphere}
f(\gamma) = 
\sum_{\ell=0}^L
\sum_{m=-\ell}^\ell
\alpha_{\ell,m}
Y_{\ell}^{m}(\gamma),
\end{equation}
for some coefficients $\alpha_{\ell,m}$. Suppose that $f,g : \mathbb{S}^2 \rightarrow \mathbb{C}$ are bandlimited on the sphere in the sense of
\eqref{bandlimited-on-sphere} for $L \le S/2$. Then,
\begin{equation}  \label{clenshaw-curtis-equality}
\sum_{s=0}^S \sum_{t = 0}^{S-1} f(\gamma_{s,t}) 
\overline{g(\gamma_{s,t})} w_s = \int_{\mathbb{S}^2} f(\gamma) \overline{g(\gamma)} \dd\sigma(\gamma),
\end{equation}
see \cite[Theorem~9.60]{plonka2018numerical} (see also \cite{driscoll1994computing,potts1998fast}). 
We note that the quadrature weights $w_s$ satisfy the bound
\begin{equation}\label{eq:bound_ws}
\begin{split}
  |w_s| &\leq \frac{4\pi}{S^2} \left( 1 + \sum_{u=1}^\infty \frac{2}{4u^2-1} \right) = \frac{8\pi}{S^2}, \quad \text{for} \quad s \in \{0,\ldots,S\},
  \end{split}
\end{equation}
where the equality follows from the fact that the sum telescopes. We also note that $w_s \geq 0$ holds, since
\begin{equation}\label{eq:ws_pos}
    w_s \geq \frac{4\pi}{S^2} \left(1- \sum_{u=1}\frac{2}{4u^2-1} \right)  = 0.
\end{equation}

{
\begin{remark}[Choice of quadrature]
In related works,  different authors use different quadrature points $\gamma_{s,t}$ and weights $w_{s,t}$ on the sphere. 
For example, another choice is Gaussian quadrature nodes in the variable $t$,
see \cite{RokhlinTygert2006}. This paper assumes the product Clenshaw-Curtis quadrature defined in \S \ref{Clenshaw-Curtis-sec} above for the sake of definiteness. However, straightforward modifications to the theory and implemented code could be performed to extend our results to other quadrature choices.
\end{remark}
}

\subsection{Spherical harmonic transforms} \label{sec:spherehartransform}
We use two types of spherical harmonics transforms
between coefficients $ (\beta_{\ell, m})$ and function values $(a_{s,t})$, where  $(\ell,m)$ range over 
\begin{equation} \label{eq:ellmrange}
(\ell,m) \in \{ (\ell,m) \in \mathbb{Z}^2 : 0 \le \ell \le L, \text{ and } -\ell \le m \le \ell\}, 
\end{equation}
and $(s,t)$ range over
\begin{equation} \label{eq:strange}
(s,t) \in \{0,\ldots,S\} \times \{0,\ldots,S-1\}.
\end{equation}
The first kind of  spherical harmonics transform  maps coefficients $(\beta_{\ell, m})$ to function values $(a_{s,t})$ via
\begin{equation}\label{eq:def_sph_transform1}
(\beta_{\ell, m}) \mapsto  a_{s,t} := \sum_{\ell =0 }^L \sum_{m = -\ell}^{\ell} \beta_{\ell, m} Y_\ell^m (\gamma_{s,t}) , 
\end{equation}
where $\gamma_{s,t}$ are defined in 
\eqref{eq:angular_grid}. The second kind of spherical harmonics transform maps function values $(a_{s,t})$ to coefficients by $(\tilde{\beta}_{\ell, m})$ via
\begin{equation}\label{eq:def_sph_transform2}
(a_{s,t }) \mapsto \tilde{\beta}_{\ell, m} := \sum_{s=0}^S \sum_{t=0}^{S-1} a_{s,t} \overline{Y_\ell^m (\gamma_{s,t})} w_{s}, 
\end{equation}
where $w_{s}$ are the quadrature weights defined in \eqref{eq:clenshaw_curtis_weights}.

\subsection{Fast spherical harmonic transforms} \label{sec:fastsphehartransform}
If the number of coefficients described by the index sets
\eqref{eq:ellmrange} and \eqref{eq:strange} are both order $N$, then directly computing the transforms $(\beta_{\ell, m}) \mapsto (a_{s,t})$ and $(a_{s,t}) \mapsto 
(\tilde{\beta}_{\ell, m})$ defined by
\eqref{eq:def_sph_transform1}
and \eqref{eq:def_sph_transform2} would require order $N^2$ operations. Fast spherical harmonic transforms like 
\cite{RokhlinTygert2006,slevinsky2019fast} compute these transforms in 
$\mathcal{O}(N (\log N)^2)$ operations. We note that in 
practice, however, there are $\mathcal{O}(N^{3/2} \log N)$ implementations  \cite{bonev2023spherical} that have better running time constants and thus are more practical for small to moderate input sizes. The code accompanying our paper allows the user to choose between asymptotically fast \cite{slevinsky2019fast} and asymptotically slow (but in practice fast) \cite{bonev2023spherical} spherical harmonics transforms at runtime.

\end{appendix}

\end{document}